\documentclass[oneside]{amsart}

\usepackage{amsthm}
\usepackage{amsmath}
\usepackage{amssymb}
\usepackage{enumerate}
\usepackage{graphicx}
\usepackage[hidelinks,pagebackref,pdftex]{hyperref}
\usepackage{booktabs}
\usepackage{color}
\usepackage[dvipsnames]{xcolor}
\usepackage{import}
\usepackage{tikz-cd}
\usepackage{comment}

\AtBeginDocument{%
   \def\MR#1{}
}

\usepackage{marginnote}
\long\def\@savemarbox#1#2{\global\setbox#1\vtop{\hsize\marginparwidth
  \@parboxrestore\tiny\raggedright #2}}
\renewcommand*{\backref}[1]{}
\renewcommand*{\backrefalt}[4]{
  \ifcase #1
  [No citations.]
  \or [#2]
  \else [#2]
  \fi }

\numberwithin{equation}{section}
\theoremstyle{plain}
\newtheorem{theorem}[equation]{Theorem}
\newtheorem{corollary}[equation]{Corollary}
\newtheorem{lemma}[equation]{Lemma}
\newtheorem{conjecture}[equation]{Conjecture}
\newtheorem{proposition}[equation]{Proposition}

\newtheorem*{T1}{Theorem~\ref{Thm:SignatureCorrection}}
\newtheorem*{T2}{Theorem~\ref{Thm:SignatureSlopeTwisting}}

\newtheorem*{namedtheorem}{\theoremname}
\newcommand{\theoremname}{testing}

\theoremstyle{definition}
\newtheorem{definition}[equation]{Definition}

\newcommand{\HH}{{\mathbb{H}}}
\newcommand{\RR}{{\mathbb{R}}}
\newcommand{\ZZ}{{\mathbb{Z}}}

\newcommand{\CC}{{\mathbb{C}}}
\newcommand{\QQ}{{\mathbb{Q}}}

\renewcommand{\d}{\partial}

\newcommand{\lk}{\text{lk}}
\newcommand{\sh}{\widehat{\sigma}}
\renewcommand{\Re}{\operatorname{Re}}
\renewcommand{\Im}{\text{Im}}

\newcommand{\vol}{\operatorname{vol}}

\newcommand{\slope}{\operatorname{slope}}

\newcommand{\inj}{\operatorname{inj}}
\newcommand{\cT}{\mathcal{T}}
\newcommand{\Z}{\mathbb{Z}}
\newcommand{\eps}{\varepsilon}
\newcommand{\tw}{\mathrm{tw}}
\newcommand{\cl}{\mathrm{cl}}
\newcommand{\im}{\text{Im}}
\newcommand{\Area}{\operatorname{Area}}

\begin{document}

\title[The signature and cusp geometry of hyperbolic knots]{The signature and cusp geometry \\ of hyperbolic knots}

\author{Alex Davies}
\address{DeepMind, London, UK}
\email{adavies@google.com}

\author{Andr\'{a}s Juh\'{a}sz}
\address{Mathematical Institute, University of Oxford, Andrew Wiles Building,
	Radcliffe Observatory Quarter, Woodstock Road, Oxford, OX2 6GG, UK}
\email{juhasza@maths.ox.ac.uk}

\author{Marc Lackenby}
\address{Mathematical Institute, University of Oxford, Andrew Wiles Building,
	Radcliffe Observatory Quarter, Woodstock Road, Oxford, OX2 6GG, UK}
\email{lackenby@maths.ox.ac.uk}

\author{Nenad Tomasev}
\address{DeepMind, London, UK}
\email{nenadt@deepmind.com}

\begin{abstract}
	 We introduce a new real-valued invariant called the natural slope of a hyperbolic knot in the 3-sphere, which is defined in terms of its cusp geometry. We show that twice the knot signature and the natural slope differ by at most a  constant times the hyperbolic volume divided by the cube of the injectivity radius. This inequality was discovered using machine learning to detect relationships between various knot invariants. It has applications to Dehn surgery and to 4-ball genus. We also show a refined version of the inequality where the upper bound is a linear function of the volume, and the slope is corrected by terms corresponding to short geodesics that link the knot an odd number of times.
\end{abstract}

\maketitle

\section{Introduction}

In low-dimensional topology, there are two very different types of invariant: those derived from hyperbolic structures on 3-manifolds, and those invariants with connections to 4-dimensional manifolds. Of the latter type, one of the most fundamental invariants is the signature of a knot. Our main goal in this paper is to establish a new and unexpected connection between these two fields. We will show that the cusp geometry of a hyperbolic knot in the 3-sphere encodes information about the signature of the knot.

One of the most important geometric features of a hyperbolic knot $K$ is its maximal cusp. The boundary of this cusp is a Euclidean torus that forms the boundary of a regular neighbourhood of $K$. This torus is isometric to $\mathbb{C}/ \Lambda$ for a lattice $\Lambda$ in $\mathbb{C}$. The meridian and longitude of the knot give generators $\mu$ and $\lambda$ for $\Lambda$. The parallelogram in $\mathbb{C}$ spanned by $0$, $\mu$, $\lambda$, and $\mu+ \lambda$ forms a fundamental domain for the action of $\Lambda$ on $\mathbb{C}$. We introduce a new geometric quantity called the \emph{natural slope} that measures how far this parallelogram is from being right-angled. It can be defined by the following formula:
\[
\slope(K) = \Re(\lambda/\mu).
\] 
Alternatively, natural slope can be defined as follows. Pick a geodesic on the torus $\mathbb{C}/\Lambda$ that represents a meridian. Choose any point on such a geodesic and send off a geodesic orthogonally from this point. It runs along the knot and eventually it comes back to the initial meridian; see Figure~\ref{Fig:NaturalSlope}. In doing so, it has gone along a longitude minus some number $s$ of meridians. This number $s$ is not necessarily an integer because the geodesic may return to a different point along the meridian from where it started. This real number $s$ is the natural slope of $K$. 

We remark that quantities with a resemblance to the natural slope have been defined by other authors \cite{BenardFlorensRodau, DegtyarevFlorensLecuona}. However, these other quantities do not seem to be directly related to natural slope, and none of these previous articles seems to provide a connection between hyperbolic geometry and signature.

\begin{figure}[t]
\centering
\includegraphics[width=0.4\textwidth]{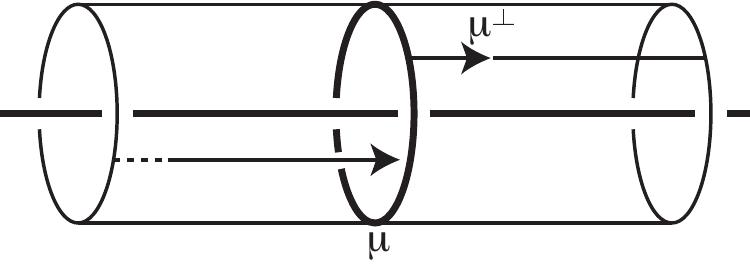}
\caption{A geodesic running in the direction $\mu^\perp$ that is perpendicular to the meridian $\mu$. By the time it returns to the meridian, it has travelled one longitude minus some multiple $s$ of the meridian. This real number $s$ is the natural slope of $K$.} \label{Fig:NaturalSlope}
\end{figure}

\begin{figure}
	\includegraphics[scale = 0.45]{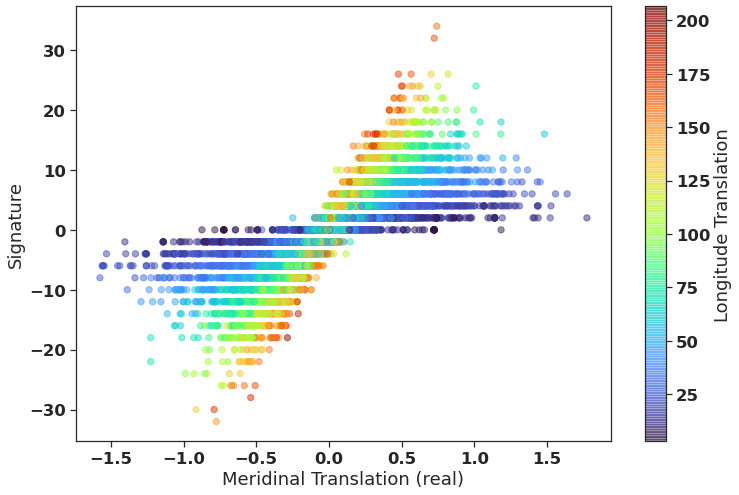}
	\caption{A plot of signature versus the real part of the meridional translation, $\Re(\mu)$, coloured by longitudinal translation, for a dataset of knots randomly generated by SnapPy.}
	\label{fig:butterfly}
\end{figure}

Experimentally, starting from the plot in Figure~\ref{fig:butterfly}, we have observed that the natural slope of $K$ is very highly correlated with $2 \sigma(K)$, where $\sigma(K)$ is the signature. See Figure~\ref{fig:signature-slope} for plots of signature versus slope for knots up to 16 crossings in the Regina census~\cite{regina} and for random knots generated by SnapPy~\cite{SnapPy} having 10 to 80 crossings in their SnapPy-simplified forms. Our goal in this paper is to prove that such a surprising connection holds and to explore its consequences. Our first main result, which we prove in Section~\ref{sec:signature}, establishes that $\slope(K)$ is approximately equal to $2 \sigma(K)$, but with an additive error that can be bounded by geometric quantities.

\begin{theorem}
	\label{thm:main}
	There exists a constant $c_1$ such that, for any hyperbolic knot $K$,
	\[
	|2\sigma(K) - \slope(K)| \leq c_1 \vol(K) \inj(K)^{-3}.
	\]
\end{theorem}

Here, $\vol(K)$ is the hyperbolic volume of the complement of $K$. Also, $\inj(K)$ is the \emph{injectivity radius} of $S^3 \setminus K$, which we define to be 
\[
\inj(K) = \inf \{\, \inj_x(S^3 \setminus K) : x \in (S^3 \setminus K) \setminus N \,\}.
\]
In the above formula, $N$ is a maximal cusp and $\inj_x(S^3 \setminus K)$ denotes the injectivity radius of a point $x$ in $S^3 \setminus K$. Note that although $\inj(K)^{-3}$ appears in the inequality in Theorem \ref{thm:main}, in practice $\inj(K)$ tends not to be particularly small. (See Figure \ref{fig:injectivity-vol} for example.)
Experimental evidence, which we provide in Section~\ref{sec:conjectures}, suggests that $c_1$ should be quite small: perhaps $c_1 = 0.3$ suffices. This is based on the largest value $0.234$ of $|2 \sigma(K) - \slope(K)| \inj(K)^3 / \vol(K)$ that we managed to obtain by studying a class of knots that are closures of certain braids.

\begin{figure}
	\includegraphics[width = 0.49\textwidth]{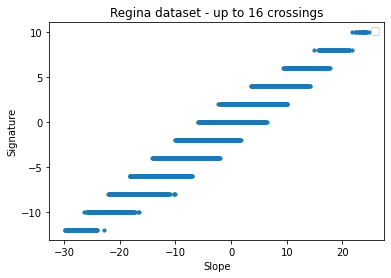}
		\includegraphics[width = 0.49\textwidth]{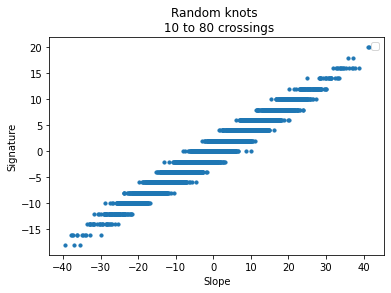}
	\caption{A plot of signature versus slope for knots up to 16 crossings in the Regina census (left) and for a dataset of knots randomly generated by SnapPy having 10 to 80 crossings in their SnapPy-simplified form (right).}
	\label{fig:signature-slope}
\end{figure}

%\marginpar{ML: Have reworded this paragraph. I have dropped mention of $\sqrt{vol(K)}$. I have also motivated the following defn.}
One might wonder whether there is a constant $c_2$ such that 
\[
|2 \sigma(K) - \slope(K)| \leq c_2 \vol(K) 
\]
for every hyperbolic knot $K$. However, we show in Corollary~\ref{cor:counterexample1} that there cannot exist such a constant. We achieve this by exhibiting a sequence of examples that are obtained by twisting 3 strands of a hyperbolic knot.
% Furthermore, there exists a hyperbolic knot $K$ with $|\sigma(K)| >  \sqrt{\vol(K)}$, but with $\sigma(K)$ and $\mathrm{Re}(\mu(K))$ having the same sign; see Corollary~\ref{cor:counterexample2}.
Nevertheless, we can estimate $\sigma(K)$ in terms of geometric quantities, with an error that is at most a linear function of $\vol(K)$. The main term in this estimate is $\slope(K)/2$, but there are also correction terms that are defined using the complex length of short geodesics. From the complex lengths, the following parameters are computed.

%\marginpar{ML:  $\im(\mathrm{cl}(\gamma))$ now lies in $(-\pi, \pi]$}
\begin{definition} \label{def:twisting}
	Let $\gamma$ be a geodesic in a hyperbolic 3-manifold with complex length $\cl(\gamma)$. Here, $\mathrm{cl}(\gamma)$ is chosen so that $\im(\cl(\gamma)) \in (-\pi, \pi]$. The \emph{twisting parameter} $\mathrm{tw}(\gamma) = (\mathrm{tw}_p(\gamma), \mathrm{tw}_q(\gamma))$ is the pair $(p,q)$ of coprime integers satisfying the following:
	\begin{enumerate}
		\item $p$ is even and $q$ is odd and non-negative;
		\item subject to this condition, the quantity $|\mathrm{cl}(\gamma) p + 2 \pi i q |$ is minimised;
		\item if there are several values of $(p,q)$ for which this quantity is minimised, then choose the one that is minimal with respect to lexicographical ordering.
	\end{enumerate}
\end{definition}

Consider a hyperbolic knot $K$ in $S^3$. For any $\eps \in \RR_+$ less than the Margulis constant $\eps_3$,
let $\mathrm{OddGeo}(\eps/2)$ denote the set of geodesics with length less than $\eps/2$ and having odd linking number with $K$.
For $p$, $q \in \ZZ_+$, the signature correction term $\kappa(p, q)$ is given by Definition~\ref{def:kappa} and satisfies 
\[
\kappa(p, q) = -\sigma(T(p,q)) - pq/2,
\]
where $T(p,q)$ is the $(p,q)$ torus knot.
%\marginpar{Added defn of $T(p,q)$}
Then we have the following refinement of Theorem~\ref{thm:main}, which we prove in Section~\ref{sec:signature-correction-proof}, that does not depend on the injectivity radius:

%\marginpar{ML: Now stated as in Section 6}
\begin{theorem}
	\label{Thm:SignatureCorrection}
	Let $\eps_3$ be the Margulis constant, and let $\eps \in (0, \eps_3)$. Then
	there is a constant $c_4$ (depending on $\eps$) such that for any hyperbolic knot $K$, the quantities $\sigma(K)$ and
	\[
	\slope(K)/2 - \sum_{\gamma \in \mathrm{OddGeo}(\eps/2)} \kappa(\mathrm{tw}_p(\gamma), \mathrm{tw}_q(\gamma)) 
	\]
	differ by at most $c_4 \vol(K)$. 
\end{theorem}

%\begin{theorem}
%	\label{Thm:SignatureCorrection}
%	Let $\eps \in (0, \eps_3)$,  where $\eps_3$ is the Margulis constant. Then there is a constant $c_4$ depending on $\eps$ such that for any hyperbolic knot $K$,
%	\[
%	\left| \sigma(K) - \slope(K)/2 + \sum_{\gamma \in \mathrm{OddGeo}(\eps/2)} \kappa(\tw_p(\gamma), \tw_q(\gamma)) \right| \le c_4 \vol(K).
%	\]
%\end{theorem}

Figures~\ref{fig:example1} and~\ref{fig:example2} illustrate the relationship between signature and slope  in Theorem~\ref{thm:main} for the knots $6_1$ and 12a52, respectively.

\begin{figure}[h]
	\includegraphics[width=0.3\textwidth]{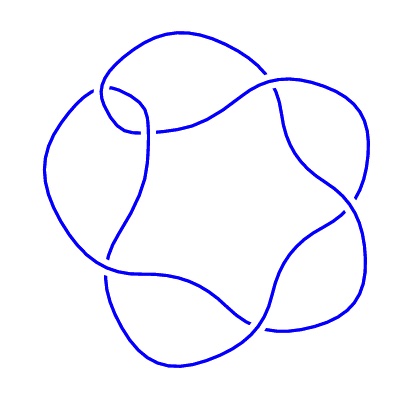}
	\includegraphics[width=0.4\textwidth]{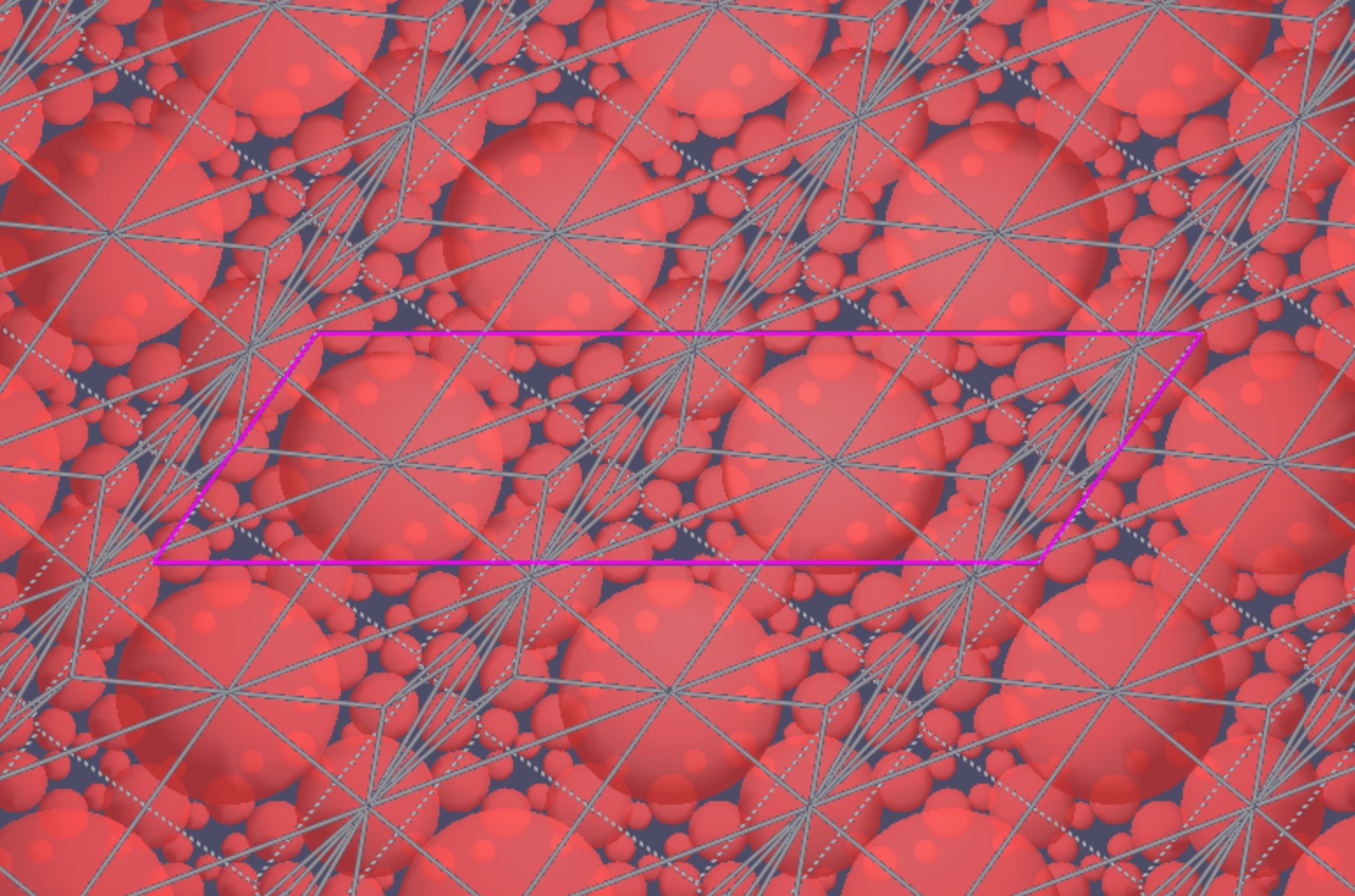}
	\caption{Left: The stevedore knot $6_1$, which is a slice knot. Right: Its cusp torus, as provided by SnapPy~\cite{SnapPy}. The longitude is $3.9279$ and the meridian is $0.7237 + 1.0160i$. Its natural slope is $1.8267$ and its signature is $0$.}
	\label{fig:example1}
\end{figure}

\begin{figure}
	\includegraphics[width=0.3\textwidth]{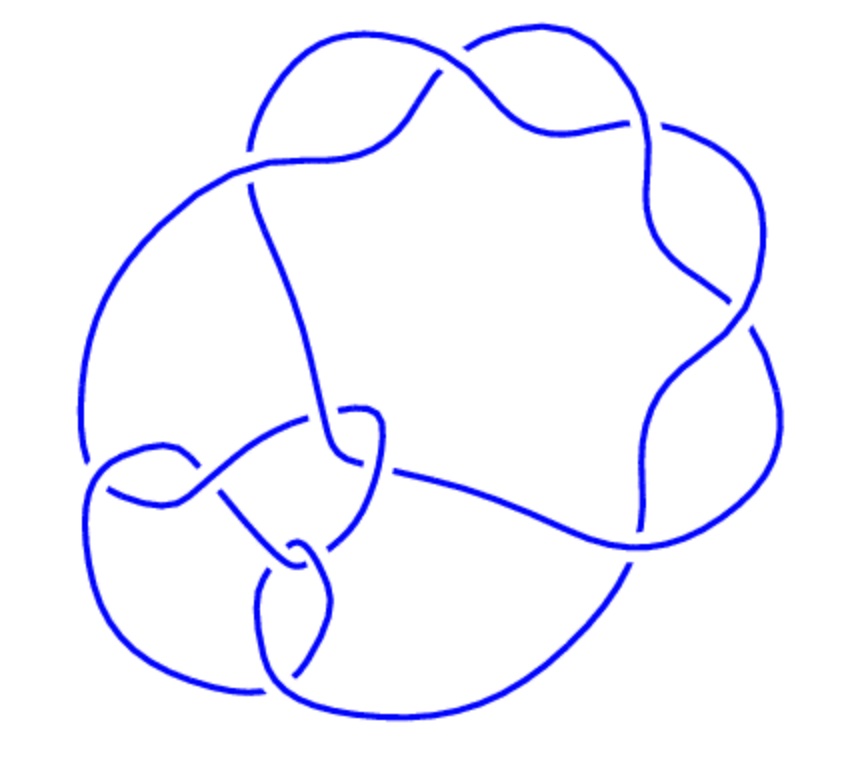}
	\includegraphics[width=0.4\textwidth]{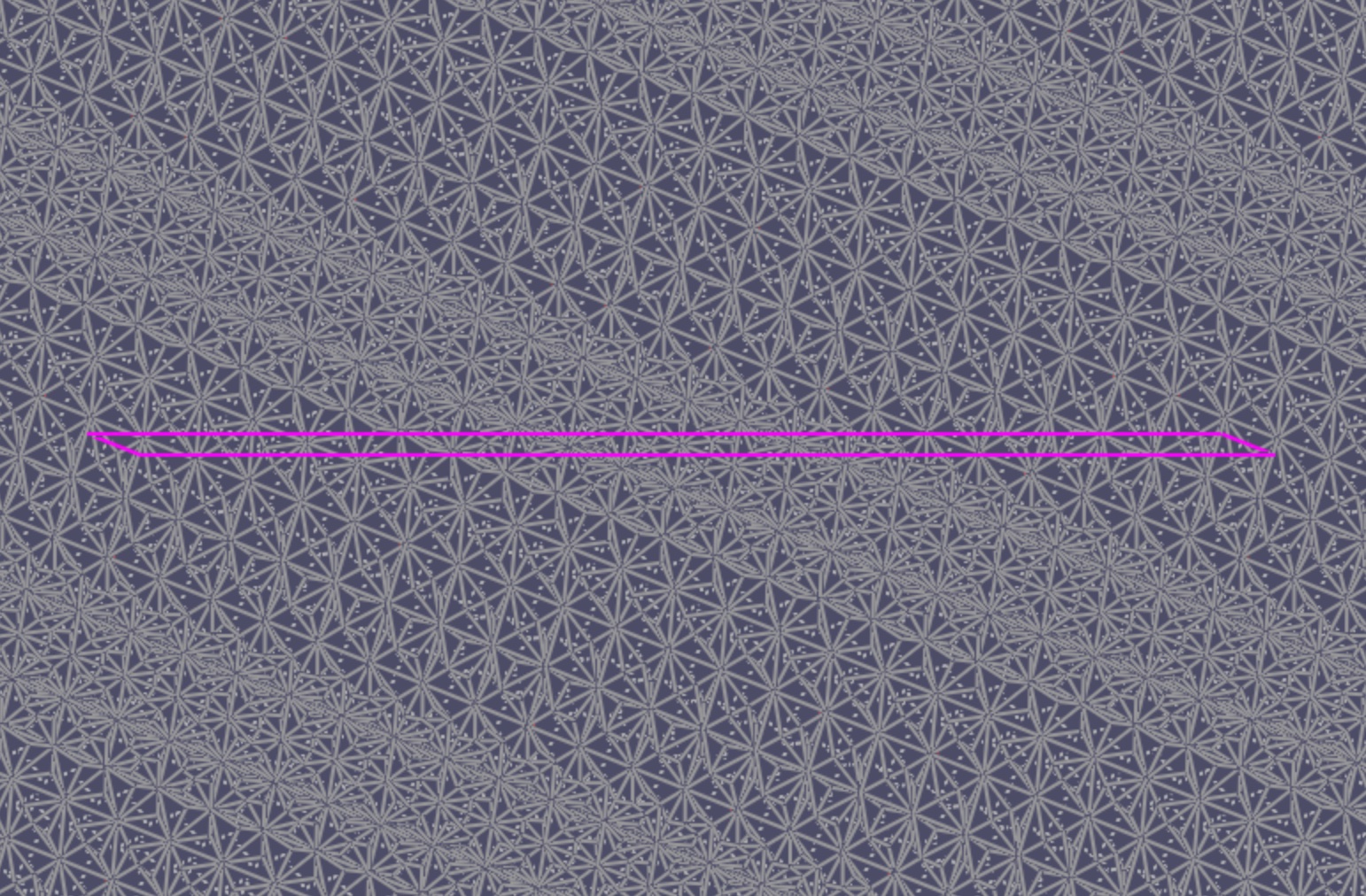}
	\caption{Left: The knot 12a52. Right: Its cusp torus. The longitude is $27.7228$ and the meridian is $-1.2838 + 0.5145i$. Its natural slope is $-18.6064$ and its signature is $-8$. Note how far the parallelogram is from being right-angled; this is the defining feature of having very positive or very negative slope.}
	\label{fig:example2}
\end{figure} 

Theorem~\ref{thm:main} has applications in low-dimensional topology. On the one hand, the signature of $K$ controls the cusp shape, which in turn has consequences for the possible exceptional surgeries on $K$. On the other hand, the cusp shape controls the signature, which has consequence for the 4-ball genus of $K$. We now provide these applications.

\subsection{An application to Dehn surgery}

Cusp geometry is well known to control the exceptional surgeries on a knot $K$. Recall that a slope $s$ on $\partial N(K)$ is said to be \emph{exceptional} if the manifold $K(s)$ obtained by Dehn filling along $s$ does not admit a hyperbolic structure. 

The \emph{length} of a slope $s = q/p \in \QQ$, denoted $\ell(s)$, is defined to be the length of any geodesic representative of $s = p \lambda + q \mu$ in the boundary of the maximal cusp. A theorem of Agol~\cite{Agol-Dehn} and Lackenby~\cite{Lackenby-Dehn} states that if $\ell(s) > 6$, then $s$ is not exceptional.

We relate slope length to natural slope, using the following simple geometric lemma, which we will prove in Section \ref{sec:slope}.

\begin{lemma}\label{lem:length-slope}
If $K$ is a hyperbolic knot, then the length of the slope $q/p$ satisfies
\[
\ell(q/p) \geq |p \slope(K) + q|.
\]
Hence, if $q/p$ is exceptional, then
\[
q/p \in [-\slope(K) - 6/p, -\slope(K) + 6/p].
\]
\end{lemma}

Given that $\slope(K)$ and $2 \sigma(K)$ are highly correlated, one would therefore expect that any exceptional slope $q/p$ should lie within a short interval around $-2 \sigma(K)$. It is also known that $|p| \le 8$, by a theorem of Lackenby and Meyerhoff~\cite{LackenbyMeyerhoff}. Hence, we obtain a bounded set of slopes that contains all the exceptional ones, and that is defined in terms of the signature. 

An interesting case is the $(-2,3,7)$-pretzel knot $12n242$. This has signature $-8$ and slope approximately $-18.215$. It has $7$ exceptional slopes: $16$, $17$, $18$, $37/2$, $19$, and $20$. Observe that these slopes are concentrated in a short interval $[16,20]$ that contains both $-\slope(K)$ and $-2 \sigma(K)$.
This close correlation between the exceptional slopes and $-2 \sigma(K)$ seems to be a phenomenon that had not previously been observed. Specifically, we have the following consequence of our main theorem. 

\begin{corollary}
If $K$ is a hyperbolic knot and $q/p$ is a slope satisfying
\[
|q/p + 2\sigma(K)| > \left ( 6/|p| \right) + c_1 \vol(K) \inj(K)^{-3} \quad \text{or} \quad |p| > 8,
\]
then the manifold $K(q/p)$ obtained by $q/p$ Dehn surgery along $K$ is hyperbolic.
\end{corollary}

Theorem~\ref{Thm:SignatureCorrection} gives a similar bound on slopes resulting in hyperbolic surgeries that does not involve $\inj(K)$. 

\subsection{An application to 4-ball genus}

One of the most important 4-dimensional quantities associated to a knot $K$ is its \emph{4-ball genus} $g_4(K)$. This is defined to be the minimal possible genus of a smoothly embedded compact orientable surface in the 4-ball $B^4$ with boundary $K$. One can also define the \emph{topological 4-ball genus} $g_4^{\mathrm{top}}(K)$ by considering locally-flat topologically embedded compact orientable surfaces with boundary $K$. The inequality $g_4(K) \geq g_4^{\mathrm{top}}(K)$ is immediate.

The following result provides a lower bound on $g_4^{\mathrm{top}}(K)$ in terms of purely hyperbolic data. This follows immediately from our main theorem together with the well-known inequality $g_4^{\mathrm{top}}(K) \geq |\sigma(K)|/2$. 

\begin{corollary}
The topological 4-ball genus $g_4^{\mathrm{top}}(K)$ of a hyperbolic knot $K$ satisfies
\[
g_4^{\mathrm{top}}(K) \geq |\slope(K)| /4  - (c_1/4)  \vol(K) \inj(K)^{-3}.
\]
\end{corollary}

This corollary seems to be the first time that information about the 4-ball genus has been obtained in terms of hyperbolic geometry. Again, Theorem~\ref{Thm:SignatureCorrection} gives a similar lower bound on $g_4^{\mathrm{top}}(K)$ that does not involve $\inj(K)$.

\subsection{Spanning surfaces}

Theorem \ref{thm:main} is proved using a new construction of spanning surfaces with a specified slope. It is of independent interest.

\begin{theorem}\label{thm:crosscap}
	There is a constant $c_3$ such that every hyperbolic knot $K$ in $S^3$ has an unoriented spanning surface $F$ satisfying
	\[
	|\chi(F)| \leq c_3  \vol(K) \inj(K)^{-3}.
	\]
	Moreover, the boundary slope of this surface is $n/1$, where $n$ is an even integer that is closest to $\slope(K)$.
\end{theorem}

We prove this in Section~\ref{sec:crosscap}.
The \emph{crosscap number} of a knot $K$ is the minimum of $b_1(F)$ for $F$ an unoriented spanning surface of $K$. When $K$ is hyperbolic, the above theorem gives an upper bound on a version of the crosscap number where $\d F$ has slope $n/1$. 

Theorem \ref{thm:main} is proved by combining this result with a theorem of Gordon and Litherland \cite{Gordon-Litherland}, which asserts that one can compute the signature of a knot $K$ using any spanning surface $F$ for $K$; see Theorem~\ref{thm:Gordon-Litherland}.

Note that slope also gives a lower bound on the Seifert genus: 
\[
\frac{1}{4\pi} |\slope(K)| + \frac12 \le g(K); 
\]
see Proposition~\ref{prop:genus-bound}.

\subsection{Highly twisted knots}

%Given a hyperbolic knot $K$, we define its \emph{normalised signature} as
%\[
%\sh(K) := \frac{\sigma(K)}{\sqrt{\vol(K)}}.
%\]
In Section~\ref{sec:highly-twisted}, we show the following result for highly twisted knots:

\begin{theorem}
	\label{Thm:SignatureSlopeTwisting}
	Let $K$ be a knot in the 3-sphere, and let $C_1, \dots, C_n$ be a collection of disjoint simple closed curves in the complement of $K$ that bound disjoint discs. Suppose that $S^3 \setminus (K \cup C_1 \cup \dots \cup C_n)$ is hyperbolic. Let $K(q_1, \dots, q_n)$ be the knot obtained from $K$ by adding $q_i$ full twists to the strings going through $C_i$, for each $i \in \{1, \dots, n\}$. Let $\ell_i$ be the linking number between $C_i$ and $K$, when they are both given some orientation. Suppose that $\ell_1, \dots, \ell_m$ are even and $\ell_{m+1}, \dots, \ell_n$ are odd. Then there is a constant $k$, depending on $K$ and $C_1, \dots, C_n$, such that the following hold, provided each $|q_i|$ is sufficiently large:
	\[
	\left | \slope(K(q_1, \dots, q_n)) + \sum_{i=1}^n \ell_i^2 q_i \right | \leq k;
	\]
	\[
	\left | \sigma(K(q_1, \dots, q_n)) + \left (  \frac{1}{2} \sum_{i=1}^m \ell_i^2 q_i + \frac{1}{2} \sum_{i=m+1}^n (\ell_i^2 -1) q_i \right ) \right | \leq k.
	\]
\end{theorem}

The slight difference between the behaviour of $\sigma(K(q_1, \dots, q_n))$ and the behaviour of $\slope(K(q_1, \dots, q_n))/2$ as the $q_i$ tend to infinity enables us to construct families of knots that show the injectivity radius cannot be dropped from Theorem~\ref{thm:main}.
%\marginpar{ML: Have reworded and dropped mention about the signs of natural slope and $\Re(\mu)$, because it is not motivated.}
%, and where the natural slope and $\Re(\mu)$ have the same sign.

\subsection{Methodology}
One of the novel aspects of this work was the use of machine learning. We embarked with the aim of discovering new relationships between various 3-dimensional invariants. By using machine learning, we observed an unexpected non-linear relationship between $\sigma(K)$ and $\Re(\mu)$, the real part of the meridional translation $\mu$. This led us to define the natural slope, which we observed to have a strong linear correlation with $\sigma(K)$. Theorems~\ref{thm:main}  and~\ref{Thm:SignatureCorrection} are the results of our attempts to prove this correlation.

\section{Hyperbolic knots and natural slope}\label{sec:slope}

A knot $K$ is hyperbolic if its complement $S^3 \setminus K$ admits a complete finite-volume hyperbolic metric.
By the Mostow rigidity theorem~\cite{Mostow}, the hyperbolic structure is unique up to isometry, hence every geometric invariant of the hyperbolic structure on $S^3 \setminus K$ is a topological invariant of the knot. For example, the volume $\vol(K) := \vol(S^3 \setminus K)$ and the injectivity radius $\inj(K)$ defined in the introduction are such invariants.

For a pair of coprime integers $p$, $q$, the \emph{torus knot} $T(p, q)$ is one that can be drawn on the surface of the standard torus in the 3-sphere, and winds $p$ times in the longitude direction and $q$ times along the meridian.
Given a knot $K$ in $S^3$ and a knot $K'$ in the solid torus $S^1 \times D^2$, one can form the satellite of $K$ with pattern $K'$ by mapping the solid torus in a neighbourhood of $K$, and considering the image of $K'$. By the work of Thurston~\cite{Morgan-hyperbolization}, a knot is hyperbolic if and only if it is not a torus knot or a satellite knot. In particular, every hyperbolic knot is prime; i.e., not the connected sum of two non-trivial knots. In other words, one can build all knots from hyperbolic knots and torus knots using satellite operations.

\begin{definition}
For any hyperbolic knot $K$, the end of $S^3 \setminus K$ has a neighbourhood called a \emph{cusp}. The boundary $\d N$ of a maximal cusp neighbourhood $N \subset S^3 \setminus K$ is a Euclidean torus. Identify $\d N$ with $\CC / \Lambda$, where $\CC$ is the complex plane and $\Lambda$ is a lattice in $\CC$. We arrange this identification so that the longitude lifts to a straight line in $\CC$ starting at $0$ and ending at some $\lambda \in \RR_{> 0}$. This is the knot's \emph{longitudinal translation}. Given this normalisation, the meridian lifts to a straight line starting at $0$ and ending at some complex number $\mu$ with $\Im(\mu) > 0$. This is the \emph{meridional translation} of $K$.
\end{definition}

We remark that the real part of meridional translation $\Re(\mu)$ in the KnotInfo~\cite{knotinfo} data set for knots with at most 12 crossings is listed without signs. However, SnapPy~\cite{SnapPy} does compute the sign for hyperbolic knots.
%\marginpar{ML: Slight rewording here}

Note that $|\mu| \le 6$, where $|\mu|$ denotes the length of the meridian. Indeed, by work of Agol~\cite{Agol-Dehn} and Lackenby~\cite{Lackenby-Dehn}, Dehn filling along a slope longer than $6$ gives a hyperbolic 3-manifold, while Dehn filling along the meridian is $S^3$, which is not hyperbolic. Furthermore, any curve on the cusp torus $\d N$ has length at least $1$. In particular, $|\mu| \ge 1$.

If $S$ is an essential surface with connected boundary in a hyperbolic 3-manifold, then $\ell(\d S) \le -2\pi \chi(S)$; see Cooper--Lackenby~\cite[Theorem~5.1]{CooperLackenby} or Hass--Rubinstein--Wang \cite[Equation~(6)]{HassRubinsteinWang}. When $S$ is a Seifert surface for a knot $K$, then $\chi(S) = 1 - 2g(S)$. Hence, if $K$ is hyperbolic, then 
\begin{equation} \label{eqn:longitude-genus}
|\lambda| \le 4\pi g(K) - 2\pi, 
\end{equation}
where $g(K)$ is the Seifert genus of $K$.

For the maximal cusp neighbourhood $N$, we have
\[
\vol(\d N) = 2\vol(N) \le 2\vol(K),
\]
and $\vol(\d N) \le |\lambda| |\mu|$.
On the other hand, by a result of Lackenby and Purcell \cite{LackenbyPurcell}, there is a constant $C$ such that, for $K$ alternating,
\[
C \vol(K) \le \vol(\d N).
\]
Based on experimental data, one might ask if this also holds for random knots.

\begin{definition}\label{def:slope}
The \emph{natural slope} $\slope(K)$ of a hyperbolic knot $K$ is defined as follows. Let $\mu^\perp$ be a unit vector at the origin of $\mathbb{C}$ orthogonal to $\mu$. Then some multiple of $\mu^\perp$ is equal to $\lambda - s \mu$ for some $s \in \RR$. Then $\slope(K) := s$.
\end{definition}

\begin{lemma} \label{lem:slope-formula}
We have
\[
\slope(K) = \Re(\lambda/\mu) = \lambda \Re(\mu) / |\mu|^2.
\]
\end{lemma}

\begin{proof}
Figure~\ref{Fig:NaturalSlopeCalc} shows a lift of the cusp torus to the complex plane $\mathbb{C}$. The point $\lambda - s \mu$ is shown (which is a multiple of $\mu^\perp)$. If we apply the transformation to $\mathbb{C}$ that is multiplication by $1/\mu$, then $\mu^\perp$ becomes purely imaginary. So $\lambda/\mu - s$ is purely imaginary. Hence, $s = \mathrm{Re}(\lambda/\mu)$. This is also equal to $\lambda \Re(\mu) / |\mu|^2$.
\end{proof}

\medskip
\begin{figure}[h]
\centering
\includegraphics[width=0.7\textwidth]{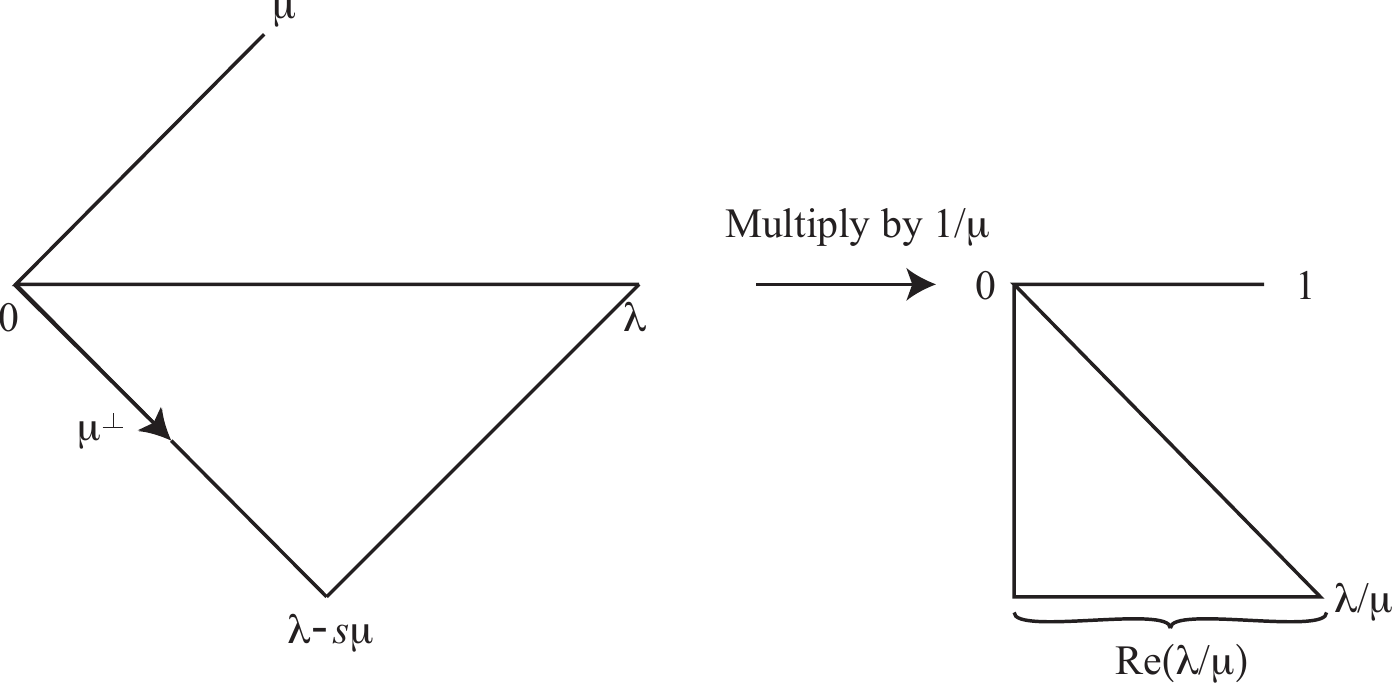}
\caption{The calculation of natural slope} \label{Fig:NaturalSlopeCalc}
\end{figure}

We are now ready to prove Lemma~\ref{lem:length-slope} from the introduction:

\begin{proof}[Proof of Lemma~\ref{lem:length-slope}]
	We have $\ell(q/p) = |p \lambda + q \mu|$. Since $\lambda \in \RR$, 
	\[
	\ell(q/p)^2 = p^2 \lambda^2 + 2pq \lambda \Re(\mu)+ q^2 |\mu|^2.
	\]
	On the other hand, by Lemma~\ref{lem:slope-formula}, we have $\slope(K) = \lambda \Re(\mu) /|\mu|^2$. Hence
	\[
	|p \slope(K) + q|^2 
	= p^2 \lambda^2 \frac{\Re(\mu)^2}{|\mu|^4} + 2pq \lambda \frac{\Re(\mu)}{|\mu|^2} + q^2 \le \ell(q/p)^2
	\]
	%\marginpar{$|\mu|^4$ instead of $|\mu|^2$}
	since $|\mu| \ge 1$.
\end{proof}

Slope gives a lower bound on the Seifert genus:

\begin{proposition}\label{prop:genus-bound}
	If $K$ is a hyperbolic knot in $S^3$, then 
	\[
	\frac{1}{4\pi} |\slope(K)| + \frac12 \le g(K).
	\]
\end{proposition}

\begin{proof}
	By equation~\eqref{eqn:longitude-genus}, we have $|\lambda| \le 4\pi g(K) - 2\pi$. Furthermore, $|\mu| \ge 1$. Together with Lemma~\ref{lem:slope-formula}, we obtain that
	\[
	|\slope(K)| = |\lambda| \frac{|\Re(\mu)|}{|\mu|^2} \le \frac{|\lambda|}{|\mu|} \le 4\pi g(K) - 2\pi,
	\]
	and the result follows.
\end{proof}

\section{Proof of Theorem~\ref{thm:crosscap}} \label{sec:crosscap}

The key to proving Theorem~\ref{thm:crosscap} is the construction of a nice triangulation of a hyperbolic knot complement:

%\marginpar{ML: also need to remove $K$ to give $M$. Replaced `the' by `a' in defn of $n$. Changed $\d M_-$ to $\d M$.}
\begin{proposition}\label{prop:triangulation}
	There is a constant $c_1$ such that, for every hyperbolic knot $K$ in $S^3$ with embedded cusp neighbourhood $N$, there is a triangulation $\mathcal{T}$ of $M := S^3 \setminus (K \cup \mathrm{int}(N))$ with the following properties:
	\begin{enumerate}
		\item\label{it:vol-bound} The number $t$ of tetrahedra of $\cT$ is at most $c_1 \vol(K) \inj(K)^{-3}$.
		\item\label{it:nu} If $n$ is a closest even integer to $\slope(K)$, then $\nu := \lambda - n \mu$ (cf.~Definition~\ref{def:slope}) is a normal curve in $\d M$ that intersects each edge at most once.
	\end{enumerate}
\end{proposition}

\begin{proof}
%\marginpar{Decreased $\eps$ by a factor of $2$. This is to make the numbers work at the end of the proof.}
	We remark that the validity of the conclusion in the proposition does not depend on the choice of embedded cusp neighbourhood $N$. We will pick $N$ as follows. Let $N_{\mathrm{max}}$ be the maximal cusp neighbourhood. Retract this to form the embedded cusp neighbourhood $N$, so each point of $\partial N$ has distance $0.5$ from $\partial N_{\mathrm{max}}$. Note that the Euclidean metric on $\partial N$ is obtained from that of $\partial N_{\mathrm{max}}$ by scaling by the factor $e^{-0.5} = 1/\sqrt{e}$.
	
	Let $\eps := \inj(K)/2$. We use a variation of J{\o}rgensen's and Thurston's method \cite[\S5.11]{Thurston-notes} to build the triangulation $\cT$. (See also the work of Breslin \cite{Breslin} and Kobayashi-Rieck \cite{KobayashiRieck}.) 
	
	We pick a maximal collection of points in $\d M$ that are all at least $\eps/8$ from each other. We will extend this to a collection of points $P$ in $M$ without adding any new points in $\d M$. Our aim is to ensure that the Voronoi diagram for $P$ in $M$ restricts to the Voronoi diagram for $P \cap \d M$ in $\d M$, where the latter is given its Euclidean metric. Recall that the Voronoi diagram~\cite{Voronoi2}\cite{Voronoi1} corresponding to $P$ is a cell structure of $M$ where the interior of every 3-cell consists of the set of points in $M$ that are closer to a specific point of $P$ than any other point of $P$. Similarly, the Voronoi diagram for $P \cap \d M$ is a cell structure of $M$ where the interior of every 2-cell consists of the set of points in $\d M$ that are closer (in the Euclidean metric) to a specific point of $P \cap \d M$ than any other point of $P \cap \d M$.
	
	The Voronoi diagram for $M$ can be constructed as follows. The universal cover $\mathbb{H}^3 \rightarrow S^3 \setminus K$ restricts to the universal cover $\tilde M \rightarrow M$. This set $\tilde M$ is obtained from $\mathbb{H}^3$ by removing the interior of the inverse image of $N$. We may arrange that one component of this inverse image is a horoball $N_\infty = \{ (x,y,z) : z \geq k \}$ in the upper half-space model for $\mathbb{H}^3$, for some $k > 0$. Let $\tilde P$ denote the inverse image of $P$ in $\tilde M$. Each cell of the Voronoi diagram for $M$ is the image of a cell for the Voronoi diagram for $\tilde P$ in $\tilde M$. Each 2-cell that does not lie in $\d \tilde M$ is equidistant from two points of $\tilde P$. Hence, it is totally geodesic. Our aim is to ensure that each such 2-cell that intersects the horosphere $\d N_\infty$ is equidistant between two points of $\tilde P \cap \d N_\infty$. This will imply that the 2-cell intersects $\d N_\infty$ in a Euclidean geodesic arc. The union of these arcs forms the 1-skeleton of the Voronoi diagram for $\tilde P \cap \d N_\infty$ in $\d N_\infty$. Thus, we can deduce that the Voronoi diagram for $P$ in $M$ restricts to the Voronoi diagram for $P \cap \d M$ in $\d M$.
	
	We now describe how the set $P$ is chosen. We have already picked a maximal collection of points in $\d M$ that are all at least $\eps/8$ from each other. This set will be $P \cap \d M$. We then add points to this set that lie in the interior of $M$, but subject to the condition that each of these points in the interior of $M$ has distance at least $\eps/4$ from the other points in the set. We stop when it is no longer possible to add any further points with this property. Let $P$ be the resulting set of points.

	By our choice of $P$, each point in $\d M$ has distance less than $\eps/8$ from some point of $P \cap \d M$. It also has distance at least $\eps/8$ from each point of $P \cap \mathrm{int}(M)$. Thus, for each point of $\partial M$, each of its closest points in $P$ also lies in $\d M$. 
	
	Now consider a 2-cell of the Voronoi diagram for $\tilde M$ that intersects $\d N_\infty$ but does not lie in $\d N_\infty$. This is equidistant between two points $p_1$ and $p_2$ of $\tilde P$. The intersection between this 2-cell and $\d N_\infty$ is an arc. Let $x$ be any point in the interior of this arc. Then $x$ is equidistant between $p_1$ and $p_2$, and these are the closest two points of $\tilde P$ to $x$. As argued above, any point of $\tilde P$ that is closest to $x$ must lie in $\d \tilde M$. We will show that, in fact, $p_1$ and $p_2$ lie in $\d N_\infty$. Suppose not. Then one of these points lies in $\d \tilde M - \d N_\infty$. The shortest arc from $x$ to $\d \tilde M - \d N_\infty$ must run through the inverse image of $\d N_{\mathrm{max}}$. One component of this inverse image is a horosphere about the point at infinity, with distance $0.5$ from $\d N_\infty$. Hence, the length of this arc is at least $0.5$. On the other hand, each point in $\d \tilde M$ has distance less than $\eps/8$ from some point of $P \cap \d \tilde M$. We will show below that $\eps/8 < 0.12 < 0.5$, and hence this is a contradiction.

	Thus, we have indeed guaranteed that the restriction to $\d M$ of the Voronoi diagram for $P$ in $M$ is the Voronoi diagram for $P \cap \d M$ in $\d M$, as claimed. We now subdivide each 2-cell of the Voronoi diagram for $M$ into triangles without introducing any new vertices, and subdivide each 3-cell into tetrahedra by coning off from the point of $P$ lying in it, obtaining the triangulation $\cT$ of $M$. Since the restriction of the Voronoi diagram to $\d M$ agrees with that arising from its Euclidean metric, this implies that each triangle of $\cT$ in $\partial M$ is straight.

	Since the open balls of radius $\eps/16$ about the points of $P$ are pairwise disjoint,
	\[
	|P| \vol(B(\eps/16)) \le \vol(S^3 \setminus K),
	\] 
	where $B(\eps/16)$ is a ball in $\HH^3$ of radius $\eps/16$.
	%\marginpar{ML:Added sentence here}
	
	We claim that the number $t_p$ of tetrahedra of $\cT$ incident to a point $p \in P$ is at most a universal constant $k$. Indeed, when $p$ lies in the interior of $M$, $t_p$ is exactly the number of triangles in the boundary 2-sphere $S$ of the 3-cell of the Voronoi diagram containing $p$. When $p$ lies in the boundary of $M$, $t_p$ is the number of triangles in this sphere that are not incident to $p$. When a vertex of one of these triangles lies in the interior of $M$, it is equidistant from at least four points of $P$, one of which is $p$. When a vertex of the triangles lies on the boundary of $M$, it is equidistant from at least three points of $P$, one of which is $p$. So, a vertex in $S$ is specified by choosing two or three other points of $P$, each of which is at most $\eps/2$ from $p$. The ball $B(p, \eps/2)$ is embedded in $S^3 \setminus K$, since $\eps/2 = \inj(K)/4$, and hence lifts to a ball $B$ in $\HH^3$. The balls of radius $\eps/16$ about the inverse image of $P$ in $B$ are disjoint, and lie within $B(9\eps/16)$. So, the number of points of $P$ at most $\eps/2$ from $p$ is bounded above by 
	\[
	k_0 := \left\lfloor \frac{\vol(B(9\eps/16))}{\vol(B(\eps/16))} \right\rfloor. 
	\]
	It follows that $t_p \le k := \binom{k_0}{3}$. Then the total number of tetrahedra
	\[
	t \le k |P| \le k \vol(K) / \vol(B(\eps/16)) \le c_1 \vol(K) \inj(K)^{-3}
	\]
	for a universal constant $c_1$.
	
	%\marginpar{Now picking the representative for $\nu$}
	We may pick the Euclidean geodesic representative for the slope $\nu$ so that it misses the vertices of $\cT$. Hence $\nu$ is a normal curve, because it is a Euclidean geodesic and each triangle of $\mathcal{T}$ in $\d M$ is straight. We now show $\nu$ does not intersect any triangle in $\d M$ more than once. Let $D$ be a fundamental domain in $\partial N_{\mathrm{max}}$ with sides $\mu$ and $\nu$. (See Figure \ref{Fig:NaturalSlopeFigure}.) We will show that the perpendicular distance $h$ between the sides of $D$ that are parallel to $\nu$ is at least 0.55. Hence, the perpendicular distance between sides of the corresponding fundamental domain in $\d N_\infty$ is at least $0.55/\sqrt{e} > 0.33$. On the other hand, we will show that the length of each edge of $\cT$ in $\d M$ is at most $0.23$. This will imply that in the triangulation of $\partial M$, no triangle can run in $D$ between these opposite sides, and hence that $\cT$ satisfies property~\eqref{it:nu}. This will complete the proof.
	
	According to a theorem of Cao and Meyerhoff \cite{Cao-Meyerhoff}, the area $A$ of the boundary of the maximal cusp is at least $3.35$. Let $\theta$ be the angle of two of the four corners of $D$ satisfying $0 < \theta \leq \pi/2$. Say that this angle is at the vertex $v_1$ of $D$, and label the remaining vertices $v_2$, $v_3$, $v_4$, so that the line joining $v_1$ to $v_2$ has slope $\mu$. 
	
	Let $b$ be the perpendicular projection of $v_4$ onto the line joining $v_1$ and $v_2$. We claim that $b$ lies between $v_1$ and $v_2$, or possibly equals one of these vertices. Place $v_1$ at the origin in the complex plane. Then $v_2 = \pm \mu$ and $v_4 = \lambda - n \mu$. Now, by the definition of $s = \slope(K)$, the perpendicular projection of $\lambda - s \mu$ onto the line through $v_1$ and $v_2$ is $v_1$. Hence, the perpendicular projection $b$ of $\lambda - n \mu$ onto this line has distance $|n-s| \, |\mu|$ from $v_1$. But $n$ is a closest even integer to $s$, and so $|n-s| \leq 1$. Therefore, $b$ lies between $v_1$ and $v_2$, or is equal to one of these points, as claimed.
%\marginpar{Added proof that $b$ lies between $v_1$ and $v_2$}	
	
\begin{figure}[h]
\centering
\includegraphics[width=0.5\textwidth]{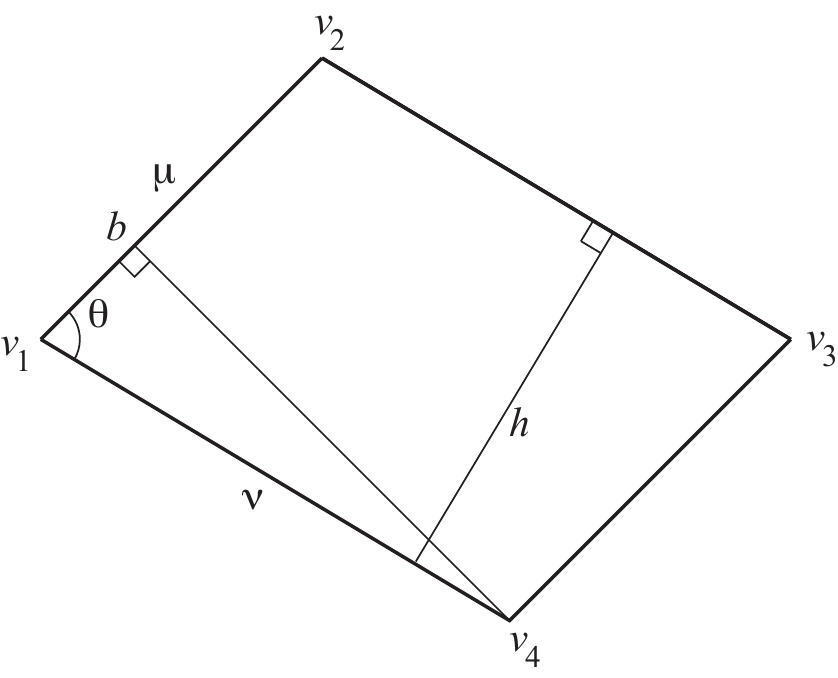}
\caption{A fundamental domain $D$ in $\partial N_{\mathrm{max}}$ with sides $\mu$ and $\nu$}
\label{Fig:NaturalSlopeFigure}
\end{figure}

	Hence,
	\[
	\tan \theta \geq A / |\mu|^2
	\]
	and so
	\[
	\sec^2 \theta = 1 + \tan^2 \theta \geq \frac{A^2 + |\mu|^4}{|\mu|^4}.
	\]
	Therefore, 
	\[
	\sin^2 \theta = 1 - \cos^2 \theta \geq 1 - \frac{|\mu|^4}{A^2 + |\mu|^4} = \frac{A^2}{A^2 + |\mu|^4}.
	\]
	So, the distance $h$ satisfies
	\[
	h = |\mu| \sin \theta \geq \frac{|\mu| A}{\sqrt{A^2 + |\mu|^4}}.
	\]
	The square of the reciprocal of this expression is
	\[
	\frac{A^2 + |\mu|^4}{|\mu|^2 A^2} = \frac{1}{|\mu|^2} + \frac{|\mu|^2}{A^2}. 
	\]
	It is easy to check that this is a convex function of $|\mu|$ and hence its maximal value over the interval $1 \leq |\mu| \leq 6$ occurs when $|\mu| = 1$ or $6$. It also is maximised by taking $A$ as small as possible, in other words $A = 3.35$. We deduce that $h$ is at least
	\[
	\frac{6 \times (3.35)}{\sqrt{(3.35)^2 + (36)^2}} \geq 0.55.
	\]
Hence, the perpendicular distance between sides of the corresponding fundamental domain in $\d N_\infty$ is at least $0.55/\sqrt{e} > 0.33$.
	
	We now compare this to the maximal length of an edge of $\cT$ in $\d M$. Each triangle of $\cT$ in $\d M$ lies within a disc centred at a point of $P \cap \d M$ with radius at most $\eps/8$. Hence each triangle has side length at most $\eps/4 = \inj(K)/8$. Now the length $L$ of the shortest slope $s$ on $\d N_{\mathrm{max}}$ is at most $|\mu| \leq 6$. This gives an upper bound on $\inj(K)$, as follows. By applying an isometry to hyperbolic space, we may arrange that a component of the inverse image of $N_{\mathrm{max}}$ in upper half space is $\{ (x,y,z): z \geq 1 \}$. We may also arrange that a covering transformation corresponding to $s$ is $(x,y,z) \mapsto (x + L, y, z)$. It therefore sends $(0,0,1)$ to $(L,0,1)$. The hyperbolic distance between these points is at most
	\[
	2 \ln \left ( \frac{6 + \sqrt{40}}{2} \right ) \leq 3.64.
	\]
	Hence, $\inj(K)$ is at most $1.82$ and $\eps/4$ is at most $0.23$. This completes the proof.
\end{proof}	

\begin{proof}[Proof of Theorem~\ref{thm:crosscap}]	
	Let the triangulation $\cT$ and the curve $\nu$ be as in Propositions~\ref{prop:triangulation}. 
	Since $\nu = \lambda - n \mu$ for $n$ even, $[\nu] = [\lambda] \in H_1(\d M; \Z_2)$, so $\nu$ bounds an unoriented surface $S$ in $M$. If we make $S$ transverse to the 1-skeleton of $\cT$, it defines a simplicial 1-cocycle $c \in C^1(M; \Z_2)$ via $c(e) = |S \cap e| \mod 2$ for each edge $e$ of $\cT$. 
	If we connect the midpoints of the edges $e$ of $T$ such that $c(e) = 1$, we obtain a surface $F$ that intersects each tetrahedron $T$ of the triangulation $\cT$ in at most one triangle or square. In particular, $F$ is a normal surface. Furthermore, $\d F = \nu$ as $\nu$ is a normal curve that intersects each triangle in $\d M$ at most once. Discard any closed components of $F$.
	
	Let $t$ be the number of tetrahedra of $\cT$. Furthermore, write $v$, $e$, and $f$ for the number of vertices, edges, and faces of $F$, respectively. By the above, $f \le t$. Then $\chi(F) = v - e + f$, and since $F$ is not a disk, $|\chi(F)| = e - f - v$. Since every face of $F$ is a triangle or a quadrilateral, 
	\[
	e \le \frac{4f + e_\d}{2} \le t + f + \frac{e_\d}{2},
	\]
	where $e_\d$ is the number of edges of $F$ in $\d M$. As $v \ge e_\d$, we obtain that 
	\[
	|\chi(F)| \le t \le c_1 \vol(K) \inj(K)^{-3},
	\]
	where the second inequality is property~\eqref{it:vol-bound} of $\cT$ in Proposition~\ref{prop:triangulation}.
\end{proof}

\section{The knot signature}\label{sec:signature}

Another fundamental knot invariant is the signature $\sigma(K)$. Given a Seifert surface $S$ for $K$; i.e., a compact, oriented, and connected surface with boundary $K$, one can define the Seifert form
\[
Q_S \colon H_1(S) \times H_1(S) \to \ZZ
\]
as follows: Given $a$, $b \in H_1(S)$, we write $b^+$ for the positive push-off of $b$ into $S^3 \setminus S$. Then $Q_S(a, b) = \lk(a, b^+)$.
If $V$ is a matrix of $Q_S$, then $\sigma(K)$ is the signature of $V + V^T$.
The signature is a 4-dimensional invariant, in the sense that it gives a lower bound on the topological 4-ball genus $g_4^{\mathrm{top}}(K)$, which is the minimal genus of a compact, oriented, locally-flat, connected surface bounded by $K$ in the 4-ball $B^4$.
%\marginpar{ML: Added locally-flat}

One can also compute the signature of a knot from unoriented surfaces using the work of Gordon and Litherland~\cite{Gordon-Litherland}. Let $F$ be an unoriented surface bounding a knot $K$ in $S^3$. Let $\{b_1, \dots, b_n\}$ be a basis of $H_1(F)$, and let $b_i'$ be the double push-off of $b_i$ into $S^3 \setminus F$. Then the \emph{Goeritz matrix} $G_F$ is an $n \times n$ symmetric matrix with $(i,j)$-th entry $\lk(b_i, b_j')$ for $i$, $j \in \{1, \dots, n\}$. Furthermore, the \emph{normal Euler number} $e(F)$ of $F$ is defined to be $-\lk(K, K')$, where $K'$ is the framing of $K$ given by $F$. Gordon and Litherland proved the following:

\begin{theorem}\label{thm:Gordon-Litherland}
	Let $F$ be an unoriented surface bounding the knot $K$ in $S^3$. Then
	\[
	\sigma(K) = \sigma(G_F) + \frac{e(F)}{2},
	\]
	where $\sigma(G_F)$ is the signature of the Goeritz matrix.
\end{theorem}

We are now ready to show how Theorem~\ref{thm:main} follows from Theorem~\ref{thm:crosscap}.

\begin{proof}[Proof of Theorem~\ref{thm:main}] 
Let $F$ be the surface provided by Theorem~\ref{thm:crosscap}, with boundary slope $\nu = \lambda - n \mu$, where $n$ is a closest even integer to $\slope(K)$. Let $G_F$ be the Goeritz matrix of $F$. Since 
\[
|\chi(F)| \leq c_1 \vol(K) \inj(K)^{-3},
\]
 we deduce that 
 \[
 b_1(F) \leq c_1 \vol(K) \inj(K)^{-3} + 1, 
 \]
and so $|\sigma(G_F)| \leq c_1 \vol(K) \inj(K)^{-3} + 1$. Therefore,
\begin{align*}
	|2\sigma(K) - \mathrm{slope}(K)| 
	&\leq |2\sigma(K) -  n| + 1 \\
	&= |2\sigma(K) + \lk(K, \nu)| + 1 \\
	&= |2 \sigma(G_F)| + 1 \\
	&\leq 2 c_1 \vol(K) \inj(K)^{-3} + 3 \\
	&\leq c_2  \vol(K) \inj(K)^{-3},
\end{align*}
for the absolute constant 
\[
c_2 := 2 c_1 + \frac{3 \cdot (1.82)^3}{2.0298} < 2 c_1 + 8.92.  
\]
%\marginpar{ML: Numbers updated}
Indeed, for any hyperbolic knot $K$, we have $\inj(K) \le 1.82$ as shown in the proof of Proposition \ref{prop:triangulation}, and $\vol(K) > 2.0298$, with the figure eight knot having the smallest volume, by Cao and Meyerhoff~\cite{Cao-Meyerhoff}.
\end{proof}

In the following definition, we introduce the signature correction $\kappa(p,q)$ for integers $p$ and $q$, which is related to the signature of the $(p,q)$-torus knot. The correction terms in Theorem \ref{Thm:SignatureCorrection} are defined in terms of $\kappa(p,q)$.
%\marginpar{ML: Have moved discussion of normalised signature to \S7. Have added motivating paragraph.}

\begin{definition}\label{def:kappa}
	For any pair of positive integers $(p,q)$, we define the \emph{signature correction} $\kappa(p,q)$ recursively as follows.
	\begin{enumerate}
		\item If $p > 2q$ and $q$ is odd, then $\kappa(p,q) = \kappa(p-2q,q) - 1$.
		\item If $p > 2q$ and $q$ is even, then $\kappa(p,q) = \kappa(p-2q,q)$.
		\item If $p = 2q$, then $\kappa(p,q) = -1$.
		\item If $q \leq p < 2q$ and $q$ is odd, then $\kappa(p,q) = -\kappa(q,2q-p) - 1$.
		\item If $q \leq p < 2q$ and $q$ is even, then $\kappa(p,q) = -\kappa(q,2q-p) - 2$.
		\item If $p < q$, then $\kappa(p,q) = \kappa(q,p)$.
	\end{enumerate}
	We extend $\kappa$ to non-zero integers $p$, $q$ by defining $\kappa(-p,q) = \kappa(p,-q) = -\kappa(p,q)$. When one of $p$ or $q$ is zero, then $\kappa(p,q) = 0$.
\end{definition}
%\marginpar{ML: Have added definition when one of $p$ or $q$ is zero.}

It is reasonably clear that this gives a well-defined value of $\kappa(p,q)$. This is because it defines $\kappa(p,q)$ uniquely when $p=q$, and when $p \not= q$, it defines $\kappa(p,q)$ in terms of some $\kappa(p',q')$ where either $q' < q$, or $q' = q$ and $p' < p$. However, the rationale for the definition comes from the following fact due to Gordon, Litherland, and Murasugi~\cite{GLM}:

\begin{theorem} 
	\label{Thm:GordonLitherlandMurasugi}
	The signature of the $(p,q)$-torus link $T(p,q)$ satisfies 
	\[
	\sigma(T(p,q))= -pq/2 - \kappa(p,q).
	\]
\end{theorem}

The signature correction $\kappa(p, q)$ arises naturally as the signature of the Goeritz form of a surface bounding the $(p,q)$-torus knot, as follows.

\begin{lemma}
	\label{Lem:TorusKnotGoeritz}
	Let $V$ be the standard solid torus in $S^3$, and let $T(p,q)$ be the curve on $\partial V$ that is the $(p,q)$-torus knot, where $p$ is even and $q$ is odd. Thus, $p$ is the winding number of $T(p,q)$ in $V$. Then there is a compact unoriented surface $F$ in $V$ with boundary $T(p,q)$, and $\sigma(G_F) = -\kappa(p,q)$ for any such $F$.
\end{lemma}

\begin{proof}
	Since $p$ is even, $T(p,q)$ is trivial in $H_1(V; \ZZ_2)$. It therefore bounds an unoriented surface $F$ in $V$. Applying Gordon and Litherland's signature formula (Theorem~\ref{thm:Gordon-Litherland}) to $F$, we deduce that 
	\[
	\sigma(T(p,q)) = \sigma(G_F) + \frac{e(F)}{2}. 
	\]
	The push-off $K'$ of $\partial F$ into $F$ has linking number $pq$ with $\d F$. To see this, observe that $K'$ is homologous in $V$ to $p$ times a core curve $\gamma'$ of $V$. 
	Similarly, $\partial F$ is homologous in the solid torus $\mathrm{cl}(S^3 \setminus V)$ to $q$ times its core curve $\gamma$, which is a meridian of $\gamma'$. Thus
	\[
	\lk(\d F, K') = pq \, \lk(\gamma, \gamma') = pq,
	\]
    hence $e(F) = -pq$. So 
    \[
    \sigma(G_F) = \sigma(T(p,q)) + pq/2 = -\kappa(p,q), 
    \]
    where the final equality is Theorem \ref{Thm:GordonLitherlandMurasugi}.
\end{proof}

\begin{lemma}
	\label{Lem:AddRowColumn}
	Let $A$ be a non-singular square matrix with real entries. Let $A_+$ be a non-singular matrix obtained from $A$ by adding a final row and final column. Then $\sigma(A_+)$ is either $\sigma(A)-1$ or $\sigma(A)+1$.
\end{lemma}

\begin{proof}
	Let $\lambda_1 \leq \dots \leq \lambda_n$ be the eigenvalues of $A$, and let $\lambda_1^+ \leq \dots \leq \lambda_{n+1}^+$ be the eigenvalues of $A_+$.
	Cauchy's interlacing theorem states that
	$$\lambda_1^+ \leq \lambda_1 \leq \lambda_2^+ \leq \dots \leq \lambda_n \leq \lambda_{n+1}^+.$$
	Hence, the number of negative eigenvalues of $A_+$ is at least the number of negative eigenvalues of $A$, and similarly 
	the number of positive eigenvalues of $A_+$ is at least the number of positive eigenvalues of $A$. 
\end{proof}

\begin{lemma}
	\label{Lem:SurfaceSolidTorus}
	Let $V$ be a solid torus embedded in $S^3$. Pick a slope $\lambda$ on $\partial V$ that has winding number $1$ in $V$. Let $K$ be the knot on $\partial V$ that has slope $p \lambda + q \mu$, where $\mu$ is the meridian of $V$, and where $p$ is even and $q$ is odd. Then $K$ bounds a compact unoriented surface $F$ in $V$ with the property that the Goeritz form $G_F$ satisfies $|\sigma(G_F) + \kappa(p,q)| \leq 2$.
\end{lemma}

\begin{proof}
	Because $p$ is even, $K$ bounds a compact surface $F$ in $V$. We may pick a basis $e_1, \dots, e_n$ for $H_1(F)$ so that $e_1, \dots, e_{n-1}$ have zero winding number around~$V$. Let $V'$ be an embedding of $V$ in $S^3$ such that $K$ is sent to $T(p,q)$. Let $F'$ be the image of $F$. 
	
	%\marginpar{ML: New explanation here. Possibly a bit too vague?}
	We claim that the Goeritz forms $G_F$ and $G_{F'}$ agree on the first $n-1$ rows and columns. To prove this, we view $V'$ as the regular neighbourhood of a standard unknot embedded in the horizontal plane. Then, up to isotopy, $V$ can be obtained from $V'$ by applying Reidemeister moves and crossing changes to this unknot. None of these moves affects the first $n-1$ rows and columns of the Goeritz form, for the following reason. Any given entry of the Goeritz form is $\lk(b_i, b_j')$ for a suitable curve $b_i$ in the surface and $b_j'$ the double push-off of another curve in the surface. When the entry of the Goeritz form lies in the first $n-1$ rows and columns, these curves $b_i$ and $b_j'$ have zero winding number around the solid torus. Hence, geometrically, $b_i$ winds an equal number of times around the solid torus in opposite directions, as does $b_j'$. So, when we perform a Reidemeister move or a crossing change to the solid torus, and we compare the resulting projections of $b_i \cup b_j'$ to the horizontal plane, the sum of the signs of the crossings between $b_i$ and $b_j'$ remains unchanged. This sum is $2 \, \lk(b_i, b_j')$. This proves the claim.

	Hence, by Lemma~\ref{Lem:AddRowColumn}, we have $|\sigma(G_F) - \sigma(G_{F'})| \le 2$. But $\sigma(G_{F'}) = -\kappa(p,q)$ by Lemma~\ref{Lem:TorusKnotGoeritz}.
\end{proof}

\section{Highly twisted knots} \label{sec:highly-twisted}

The following is Theorem~\ref{Thm:SignatureSlopeTwisting} from the introduction:

\begin{T2}
	Let $K$ be a knot in the 3-sphere, and let $C_1, \dots, C_n$ be a collection of disjoint simple closed curves in the complement of $K$ that bound disjoint discs. Suppose that $S^3 \setminus (K \cup C_1 \cup \dots \cup C_n)$ is hyperbolic. Let $K(q_1, \dots, q_n)$ be the knot obtained from $K$ by adding $q_i$ full twists to the strings going through $C_i$, for each $i \in \{1, \dots, n\}$. Let $\ell_i$ be the linking number between $C_i$ and $K$, when they are both given some orientation. Suppose that $\ell_1, \dots, \ell_m$ are even and $\ell_{m+1}, \dots, \ell_n$ are odd. Then there is a constant $k$, depending on $K$ and $C_1, \dots, C_n$, such that the following hold, provided each $|q_i|$ is sufficiently large:
	\[
	\left | \slope(K(q_1, \dots, q_n)) + \sum_{i=1}^n \ell_i^2 q_i \right | \leq k;
	\]
	\[
	\left | \sigma(K(q_1, \dots, q_n)) + \left (  \frac{1}{2} \sum_{i=1}^m \ell_i^2 q_i + \frac{1}{2} \sum_{i=m+1}^n (\ell_i^2 -1) q_i \right ) \right | \leq k.
	\]
\end{T2}

One can use this to show that the factor $\inj(K)^{-3}$ cannot simply be dropped from Theorem~\ref{thm:main} (cf.~Conjecture~\ref{conj:asymptotic} for what we expect for random knots):

\begin{corollary}\label{cor:counterexample1}
	There does \emph{not} exist a constant $c_2$ such that 
	\[
	|2 \sigma(K) - \slope(K)| \leq c_2 \vol(K)
	\]
	for every hyperbolic knot $K$.
\end{corollary}

\begin{proof} 
	Pick $n=1$ and $\ell_1 = 3$. Then $\slope(K(q_1)) \sim -9q_1$, whereas $2 \sigma(K(q_1)) \sim -8 q_1$. On the other hand $\vol(K(q_1))$ is bounded.
\end{proof}
%\marginpar{ML: Have moved discussion of $\hat{\sigma}$ to \S7}

\begin{proof}[Proof of Theorem \ref{Thm:SignatureSlopeTwisting}]
	The knot $K(q_1, \dots, q_n)$ is obtained by performing $-1/q_i$ surgery on $C_i$, for each $i \in \{1, \dots, n\}$. Let $L$ denote the link $K \cup C_1 \cup \dots \cup C_n$. By Thurston's Hyperbolic Dehn Surgery theorem, as all the $|q_i|$ tend to infinity, the hyperbolic structures on $S^3 \setminus K(q_1, \dots, q_n)$ tend in the geometric topology to the hyperbolic structure on $S^3 \setminus L$. In fact, more it true. Fix a horoball neighbourhood $N$ of the cusps of $S^3 \setminus L$ that is small enough so that the cusp torus $T$ surrounding $K$ lies in the complement of $N$. Then, if all the $|q_i|$ are sufficiently large, the inclusion $(S^3 \setminus L) \setminus N \rightarrow S^3 \setminus K(q_1, \dots, q_m)$ is a bi-Lipschitz homeomorphism onto its image, with bi-Lipschitz constants that tend to $1$ as all the $|q_i|$ tend to infinity. (See \cite{Benedetti-Petronio} for instance.)
	%\marginpar{ML: Added reference}
	
	Let $\lambda(K)$ be the longitude and $\mu(K)$ the meridian of $K$. These form a basis of the lattice $\Lambda(K)$, where the cusp torus of $K$ in $S^3 \setminus L$ is $\CC / \Lambda(K)$. Let $\gamma$ be the image of $\lambda(K)$ and $\mu$ the image of $\mu(K)$ on the cusp torus $\CC / \Lambda(K(q_1, \dots, q_n))$ of $K(q_1, \dots, q_n)$. The curves $\gamma$ and $\mu$ form a basis of the lattice $\Lambda(K(q_1, \dots, q_n))$. So, we may assume that $\gamma$ and $\mu$ are approximately constant complex numbers when $|q_i|$ are large. However, we have \emph{not} normalised the lattice so that $\gamma$ is real. We know that there is some $N \in \RR_+$ such that 
	\[
	N \mu^\perp = \gamma - s' \mu 
	\]
	for some $s' \in \mathbb{R}$. Here, $N$, $\mu^\perp$, $\gamma$, $s'$, and $\mu$ all depend on $q_1, \dots, q_n$. But $N$ and $s'$ tend to fixed real numbers as the $|q_i|$ go to infinity.
	
	The key observation is that $\gamma$ is \emph{not} necessarily the longitude $\lambda$ for $K(q_1, \dots, q_n)$. In fact, the linking number between $\gamma$ and $K(q_1, \dots, q_n)$ is $\sum_i \ell_i^2 q_i$; see Figure~\ref{Fig:Linking}. For suppose that the disc bounded by $C_i$ intersects $K$ in $p_-$ points of negative sign and $p_+$ points of positive sign. So, $\ell_i = p_+ - p_-$. Then, when we perform a full twist about $C_i$, we introduce $2(p_+ + p_-)^2$ new crossings between $\gamma$ and $K(q_1, \dots, q_n)$. Of these, $2(p_+^2 + p_-^2)$ have positive sign and $4p_+p_-$ have negative sign. So the linking number between $\gamma$ and $K(q_1, \dots, q_n)$ changes by 
	\[
	p_+^2 + p_-^2 - 2p_+ p_- = \ell_i^2.
	\]
	It follows that
	\[
	\gamma = \lambda + \left(\sum_{i=1}^n \ell_i^2 q_i \right) \mu,
	\]
	and hence 
	\[
	N \mu^\perp = \lambda - \left(s' - \sum_{i=1}^n \ell_i q_i^2 \right) \mu.
	\]
	We conclude that $\slope(K(q_1, \dots, q_n)) = s' - \sum_{i=1}^n \ell_i q_i^2$. On the other hand, there is a constant $k$ such that $|s'| \le k$ if $|q_1|, \dots, |q_n|$ are sufficiently large, which implies the first inequality of the theorem.
	
	\begin{figure}
		\centering
		\includegraphics[width=0.5\textwidth]{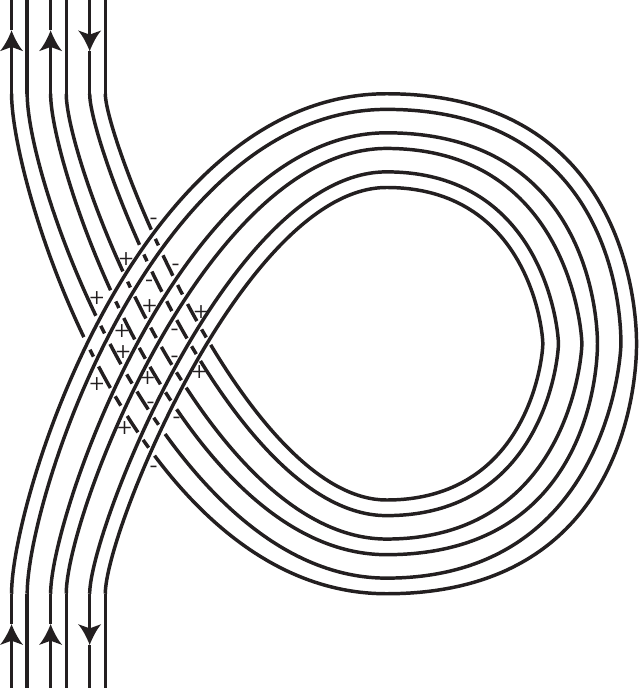}
		\caption{Each full twist about $C_i$ changes the linking number between $\gamma$ and $K(q_1, \dots, q_n)$ by $\ell_i^2$.}
		\label{Fig:Linking}
	\end{figure}
	
	Recall that $\ell_{m+1}, \dots, \ell_n$ are odd. Suppose that $q_{m+1}, \dots, q_r$ are even and that $q_{r+1}, \dots, q_n$ are odd. Let $\mu_{m+1}, \dots, \mu_r$ be meridians for $C_{m+1}, \dots, C_r$, respectively. Let $F$ be a spanning surface for 
	\[
	K \cup \mu_{m+1} \cup \dots \cup \mu_r \cup C_{r+1} \cup \dots \cup C_n.
	\]
	Since this link has even linking number with each component of $C_1 \cup \dots \cup C_r$, we may choose this spanning surface to be disjoint from these components. We can view this surface as properly embedded in the exterior of $K \cup C_1 \cup \dots \cup C_n$. It is disjoint from $\partial N(C_1) \cup \dots \cap \partial N(C_m)$. We have $F \cap \partial N(C_i) = \mu_i$ for $i \in \{ m+1, \dots, r \}$. For $i \in \{r+1, \dots, n\}$, the curve $F \cap \partial N(C_i)$ has slope equal to a longitude plus an odd number of meridians. By choosing the surface appropriately, we can ensure that this odd number is $1$. 
		%\marginpar{ML: Removed statement that $F$ has no disc, annulus of M\"obius band components}
	
	Now perform surgery along $C_1, \dots, C_n$. The surface becomes a surface in the exterior of the new link. On $\partial N(C_i)$ for $i \in \{m+1, \dots, r\}$, it now has slope equal to a meridian plus $q_i$ longitudes. On $\partial N(C_i)$ for $i \in \{r+1, \dots, n\}$, it is a meridian plus $q_i + 1$ longitudes. Since we are assuming that $|q_i|$ is sufficiently large, we can suppose that $q_i \not= 0,-1$ and hence that this slope is not meridional. Within each solid torus $N(C_{m+1}), \dots, N(C_n)$, we can now insert a surface, as shown in Figure~\ref{Fig:SpanningSurfaceTwist}. Let $F'$ denote the resulting spanning surface of $K(q_1, \dots, q_n)$.
	
	\begin{figure}[h]
		\centering
		\includegraphics[width=0.5\textwidth]{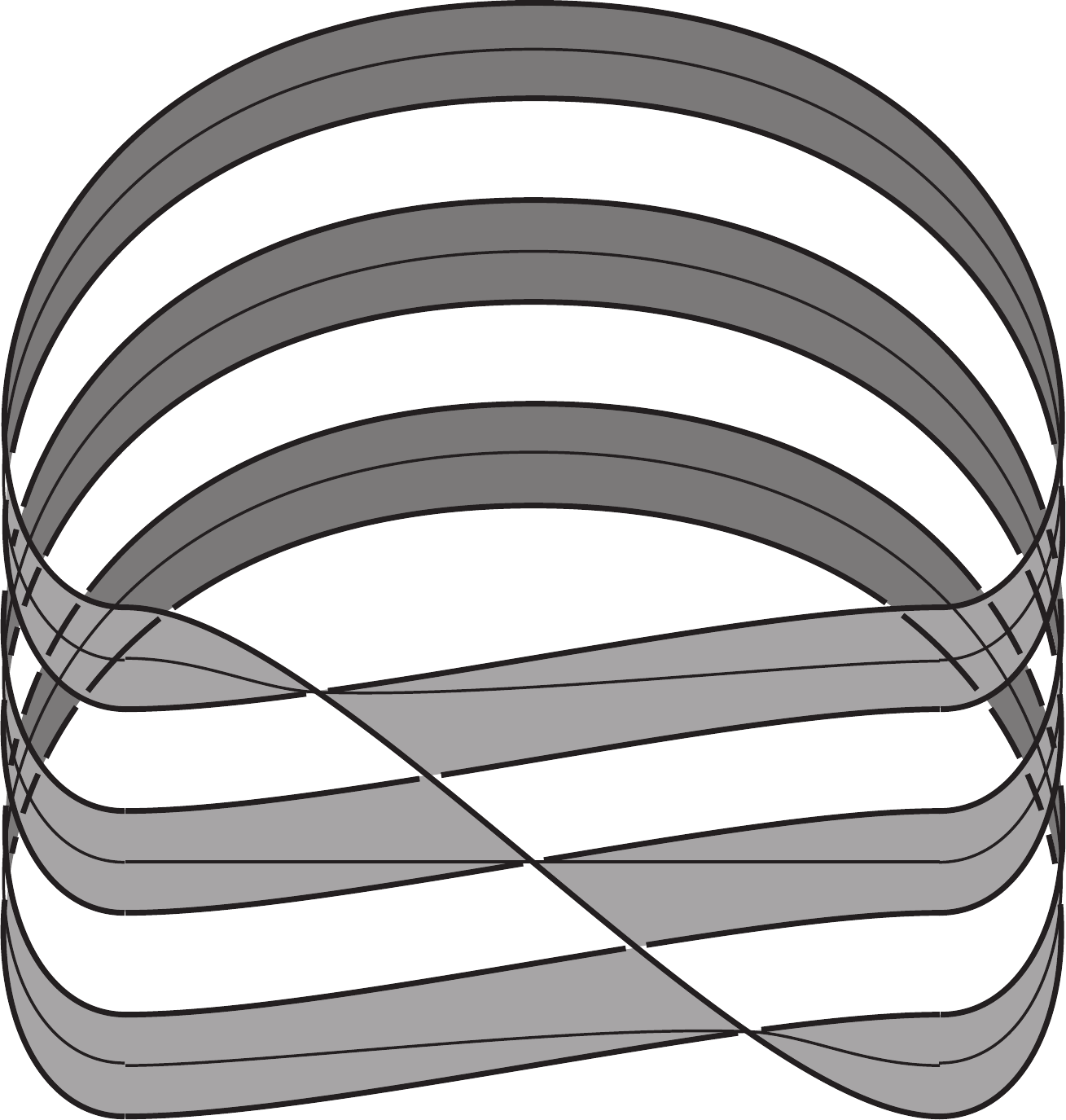}
		\caption{The part of the spanning surface in $N(C_i)$ for $i \geq m+1$. Here, $q_i = 5$ or $6$.}
		\label{Fig:SpanningSurfaceTwist}
	\end{figure}
	
	Also shown in Figure \ref{Fig:SpanningSurfaceTwist} is a collection of generators for $H_1(F' \cap N(C_i))$ for $i \geq m+1$. Note that $H_1(F' \cap N(C_i))$, for $i \geq m+1$, form direct summands of $H_1(F')$. So we can extend this set of generators to a basis of $H_1(F')$, by adding further elements of $H_1(F)$. 
%\marginpar{ML: Added brief discussion about how we can get a \emph{basis} for $H_1(F)$}
	The associated Goeritz form $G_F$ is diagonal when restricted to the rows and columns corresponding to $H_1(F' \cap N(C_1 \cup \dots \cup C_n))$. Each $C_i$ gives rise to $|q_i|/2$ diagonal entries when $m+1 \leq i \leq r$ and $|q_i + 1|/2$ entries when $r+1 \leq i \leq n$. These entries are $+1$ when $q_i$ is positive and $-1$ when $q_i$ is negative. Hence, the signature of this matrix differs from $\sum_{i=m+1}^n q_i/2$ by at most $(n-r)/2$. So, applying Lemma~\ref{Lem:AddRowColumn}, 
	\[
	\left|\sigma(G_{F'}) - \sum_{i=m+1}^n q_i/2 \right|
	\]
	is bounded.
	
	Theorem~\ref{thm:Gordon-Litherland} due to Gordon and Litherland states that 
	\[
	\sigma(K(q_1, \dots, q_n)) = \sigma(G_{F'}) + \frac{e(F')}{2}.
	\]
	Here 
	\[
	e(F') = -\lk(K(q_1, \dots, q_n), \d F') = -\lk(K, \d F') - \sum_{i=1}^n \ell_i^2 q_i.
	\]
	The second inequality of the theorem follows immediately. 
\end{proof}

\section{Proof of Theorem~\ref{Thm:SignatureCorrection}} \label{sec:signature-correction-proof}

In this section, we prove Theorem~\ref{Thm:SignatureCorrection} from the introduction:

\begin{T1}
	Let $\eps_3$ be the Margulis constant, and let $\eps \in (0, \eps_3)$. Then
	there is a constant $c_4$ (depending on $\eps$) such that for any hyperbolic knot $K$, the quantities $\sigma(K)$ and
	\[
	\slope(K)/2 - \sum_{\gamma \in \mathrm{OddGeo}(\eps/2)} \kappa(\mathrm{tw}_p(\gamma), \mathrm{tw}_q(\gamma)) 
	\]
	differ by at most $c_4 \vol(K)$. 
\end{T1}

Note that if we set $\eps = \eps_3/2$, then $c_4$ then becomes a universal constant. However, given the present uncertainty about the precise value of $\eps_3$, we do not specify $\eps$ definitively.

\begin{definition}
	\label{Def:CanonicalLongitude}
	Let $\gamma$ be an embedded closed geodesic in the hyperbolic 3-manifold $M$, and let $N(\gamma)$ be a regular neighbourhood of $\gamma$ consisting of points at most a certain distance $r$ from $\gamma$. Let $\tilde \gamma$ be a component of the inverse image of $\gamma$ in $\mathbb{H}^3$, which we can take to be $\{ (0,0,z): z > 0 \}$ in the upper half-space model. Let $N(\tilde \gamma)$ be the component of the inverse image of $N(\gamma)$ containing $\tilde \gamma$. We let $\lambda$ be the slope on $\partial N(\gamma)$ that has winding number one around $N(\gamma)$ and that lifts to a path in $N(\tilde \gamma)$ starting on the half-plane $\{ (x,y,z) : y = 0, x \geq 0 \}$ and with interior that is disjoint from the half plane $\{ (x,y,z) : y = 0, x \leq 0 \}$. In the event that this path ends precisely on the half plane $\{ (x,y,z) : y = 0, x \leq 0 \}$, $\lambda$ is chosen so that it  avoids $\{ (x,y,z) : y \leq 0, x = 0 \}$. Then $\lambda$ is called the \emph{canonical longitude} of $\gamma$. Note that it does not necessarily have zero linking number with $\gamma$.
\end{definition}
%\marginpar{ML: new definition revised since last version}

There is the following alternative interpretation of the canonical longitude in terms of the complex length of $\gamma$. 	We give $T = \partial N(\gamma)$ its inherited Riemannian metric. This is homogeneous, since any two points of $T$ differ by an isometry of $T$. The metric on $T$ therefore has constant curvature, which must be zero by the Gauss--Bonnet theorem. It is therefore Euclidean. We can represent it as the quotient of the Euclidean plane $\mathbb{E}^2$ by a lattice $\mathcal{L}$. Each slope on $T$ corresponds to a lattice point. We can assume that the lattice point corresponding to the meridian is a purely imaginary number $\mu$. As the circumference of a radius $r$ circle in the hyperbolic plane is $2\pi \sinh(r)$, we have 
	\[
	\mu = 2 \pi  \sinh(r) i, 
	\]
	where $r$ is the radius of the tube around $\gamma$. Let $\nu$ be a geodesic in $T$ that is perpendicular to a meridian and that starts and ends on the meridian (but not necessarily at the same point). Then 
	\[
	\ell(\nu) = \cosh(r) \Re(\cl(\gamma)), 
	\]
	where $\cl(\gamma)$ is the complex length of the geodesic $\gamma$ and $\Re(\cl(\gamma)) = \ell(\gamma)$; see \cite[Equation~(2.2)]{tubes}. Then the canonical longitude of $T$ is
	\[
	\lambda = \cosh(r) \Re(\cl(\gamma)) +  \sinh(r) \Im(\cl(\gamma)) i.
	\] 
	
The significance of the twisting parameter arises from the following lemma.

\begin{lemma}
	\label{Lem:TwistSlopeShort}
	Let $M$ be a hyperbolic 3-manifold and $\eps \in (0, \eps_3)$. Let $\gamma$ be a geodesic in $M$ with $\ell(\gamma) < \eps/2$. Let $T$ be the toral boundary component of $M_{(0,3\eps/4]}$ that encloses $\gamma$, let $\mu \subset T$ be a meridian of $\gamma$ and let $\lambda$ be the canonical longitude. If $(p,q) = (\mathrm{tw}_p(\gamma), \mathrm{tw}_q(\gamma))$, then
	\[
	\ell(p \lambda + q \mu) \le c_5 \Area(T)
	\] 
	for some constant $c_5$ depending only on $\eps$.
\end{lemma}
%\marginpar{ML: $\lambda$ is now canonical}

\begin{proof}
	By the Margulis lemma, the component $V$ of $M_{(0,3\eps/4]}$ containing $T$ is a solid torus, with $\gamma$ as a core curve. We claim that the tube radius $r$ of $V$ satisfies $r > \eps/8$. Indeed, note that $\gamma$ has length $\ell(\gamma) < \eps/2$, whereas at each point $y \in T$, the open ball $B(y, 3\eps/8)$ is embedded. If $r < 3\eps / 8$ and $x \in \gamma$ satisfies $d(x, y) = r$, then $B(x, 3\eps/8 - r) \subset B(y, 3\eps/8)$ is an embedded ball about $x$. So 
	\[
	3\eps / 8 - r < \ell(\gamma) / 2 < \eps /4, 
	\]
	and hence $r > \eps / 8$, as claimed.

	Suppose that $\gamma_0$  is a shortest geodesic on $T$, and let $L := \ell(\gamma_0)$. We claim that 
	\[
	L \in [k_0, k_0'] 
	\]
	for constants $k_0$, $k_0' \in \RR_+$ depending only on $\eps$. Since $T \subset \d M_{(0, 3\eps / 4]}$, every point $p \in T$ has two lifts to $\HH^3$ that are exactly $3\eps/4$ apart, and no two lifts of $p$ are less than $3\eps/4$ apart. The meridian of $T$ has length 
	\[
	\ell(\mu) = 2 \pi \sinh(r) > 2 \pi \sinh(\eps/8). 
	\]
	If $s$ is a slope different from the meridian, then $[s] = m [\gamma] \in \pi_1(M)$ for $m \neq 0$. As $[\gamma]$ has infinite order in $\pi_1(M)$, the lift $\tilde{s}$ of $s$ to $\HH^3$ satisfies $\tilde{s}(0) \neq \tilde{s}(1)$. Then 
	\[
	\ell(s) \ge d_{\HH^3}(\tilde{s}(0), \tilde{s}(1)) \ge 3\eps/4, 
	\]
	so we can set $k_0 := \min(3\eps/4, 2 \pi \sinh(\eps / 8))$.
	
	We now give an upper bound on $L$. Let $s$ be a slope on $T$ whose lift $\tilde{s}$ to $\HH^3$ satisfies $d_{\HH^3}(\tilde{s}(0), \tilde{s}(1)) = 3\eps/4$. This is again possible since $T \subset \d M_{(0, 3\eps/4]}$. If $r \le 2 \eps$, then $L \le |\mu| \le 2 \pi \sinh(2 \eps)$. Now suppose that $r > 2 \eps$. Let $N(\gamma) \subset V$ be a regular neighbourhood of $\gamma$ of radius $r-\eps$. Let $\tilde{\beta}$ be a geodesic in $\HH^3$ connecting $\tilde{s}(0)$ and $\tilde{s}(1)$, and let $\beta$ be its projection to $M$. Then $\beta$ is a geodesic homotopic to $s$ of length $3\eps/4$, which hence lies in $V \setminus N(\gamma)$. The nearest point projection $\varphi \colon V \setminus N(\gamma) \to T$ satisfies $\ell(\varphi(\beta)) \le l_0 \ell(\beta) = l_0 (3\eps/4)$ for a constant $l_0$ depending only on $\eps$. Hence 
	\[
	L \le k_0' := \max(2 \pi \sinh(2 \eps), 3l_0 \eps/4), 
	\]
	as claimed. 
	
	A consequence of $L \ge k_0$ is that $\Area(T) \ge a_0$ for a constant $a_0$ depending only on $\eps$. Indeed, a disc $D$ of radius $L/2$ on $T$ about an arbitrary point of $T$ is embedded, so 
	\[
	\Area(T) \ge \Area(D) = (L/2)^2 \pi \ge (k_0 / 2)^2 \pi =: a_0.
	\]
	
	We claim the length of the shortest curve in any nontrivial class
	in $H_1(T; \ZZ_2)$ is at most $k_1 \Area(T)$ for a constant $k_1$ depending on $\eps$. Indeed, let $\gamma_0^\perp \colon I \to T$ be a geodesic arc starting and ending on the shortest geodesic $\gamma_0$ and orthogonal to it.
	Then $\ell(\gamma_0^\perp) = \Area(T)/L$. The points $\gamma_0^\perp(0)$ and $\gamma_0^\perp(1)$ divide $\gamma_0$ into two arcs, one of which has length at most $L/2$. Let $\gamma_1$ be a geodesic representative of the closed curve that runs along $\gamma_0^\perp$ and then along the shorter of the two arcs in $\gamma_0$. 
	%\marginpar{ML: Now no longer claiming that $\gamma_0$ and a second shortest geodesic generate $H_1(T; \ZZ_2)$, although this is probably true.}
	We obtain that 
	\[
	\ell(\gamma_1) \le L/2 + \Area(T)/L. 
	\]
	The curves $\gamma_0$ and $\gamma_1$ give a basis for $H_1(T; \ZZ_2)$. Hence, the shortest representative of every nontrivial class in $H_1(T; \ZZ_2)$ is at most $L + (L/2 + \Area(T)/L)$. As $L \in [k_0, k_0']$ and $\Area(T) \ge a_0$, we have 
	\[
	L + (L/2 + \Area(T)/L) \le k_1 \Area(T)
	\]
	for $k_1 := \frac32 \frac{k_0'}{a_0} + \frac{1}{k_0}$. Indeed,
	\[
	\frac32 L \le \frac32 k_0' = \left (k_1 - \frac{1}{k_0} \right ) a_0 \le \left(k_1 - \frac1L \right) \Area(T).
	\]
	So there is some slope $(a, b)$ on $T$ with $a$ even and $b$ odd such that
	\[
	\ell(a \lambda + b \mu) \le k_1 \Area(T)
	\]
	for some constant $k_1$ depending on $\eps$.
	
	Let $T'$ be the torus obtained from $T$ by scaling by $\tanh(r)$ in the $\nu$-direction. As $r > \eps / 8$, we have $\tanh(r) \in (\tanh(\eps / 8), 1)$. Since $\tanh(r) < 1$, the shortest slope $(p, q)$ on $T'$ with $p$ even and $q$ odd has length at most $k_1 \Area(T)$.  The lattice that specifies $T'$ is generated by 
	\[
	\lambda' := \tanh(r) \cosh(r) \Re(\cl(\gamma)) + \sinh(r) \Im(\cl(\gamma)) i = \sinh(r) \cl(\gamma) 
	\]
	and $\mu = 2 \pi \sinh(r) i$. So 
	\[
	\ell(p \lambda' + q \mu) = | \cl(\gamma) p + 2 \pi i q | |\sinh(r)|. 
	\]
	Hence, by Definition~\ref{def:twisting}, the slope $p \lambda' + q \mu$ on $T'$ is the shortest among slopes for which $p$ is even and $q$ is odd.  Therefore, its length on $T'$ is at most $k_1 \Area(T)$. So
	\[
	\ell(p \lambda + q \mu) \le (k_1 / \tanh(r)) \Area(T) < (k_1 / \tanh(\eps / 8)) \Area(T). 
	\]
	So we can set $c_5 := k_1 / \tanh(\eps / 8)$, which concludes the proof of the lemma.
\end{proof}

\begin{proof}[Proof of Theorem \ref{Thm:SignatureCorrection}]
	We claim that we can build a triangulation $\mathcal{T}$ of $M_{[3\eps/4,\infty)}$ with the following properties:
	\begin{enumerate}
		\item The number of tetrahedra of $\mathcal{T}$ is at most $c \vol(K)$, where $c$ depends on $\eps$.
		\item If $n$ is a closest even integer to $\slope(K)$, then some Euclidean geodesic with slope $\lambda - n \mu$ on $\partial N(K)$ is a normal curve in $\d M_{[3\eps/4,\infty)}$ that intersects each edge of $\mathcal{T}$ at most once.
		\item\label{it:short-edges} On the component $T$ of $\d M_{[3\eps/4,\infty)}$ corresponding to $\partial N(K)$, the edges of $\mathcal{T}$ are Euclidean geodesics with length at most $\eps/15$.		
	\end{enumerate}
	%\marginpar{ML: The edge length has changed, and the following proof is new}

	We follow the construction in the proof of Proposition~\ref{prop:triangulation}, but with different constants. We pick a maximal collection of points in $\d M_{[3\eps/4,\infty)}$ that are at least $\eps/30$ apart. We then add points to this collection in the interior of $M_{[3\eps/4,\infty)}$ that have distance at least $\eps/15$ from each other and from the earlier points. We stop when it is not possible to add any further points, and denote the resulting collection by $P$. We then form the associated Voronoi diagram, subdivide the 2-cells of this cell structure into triangles without adding any new vertices, and then triangulate each 3-cell by coning from the relevant point of $P$. Let $\mathcal{T}$ be the resulting triangulation of $M_{[3\eps/4,\infty)}$.
	
	Exactly the same argument as in the proof of Proposition~\ref{prop:triangulation} gives that the number of tetrahedra of $\mathcal{T}$ is at most $c \vol(K)$, where $c$ depends on $\eps$. The length of each edge in $\d M_{[3\eps/4,\infty)}$ is now at most $\eps/15$, because we took points that were at least $\eps/30$ apart, rather than at least $\eps/8$ apart. Thus, all that needs to be proved are that the edges of $\mathcal{T}$ in $T$ are Euclidean geodesics and that there is a Euclidean geodesic with slope $\lambda - n \mu$ on $\partial N(K)$ which is a normal curve in $\d M_{[3\eps/4,\infty)}$ that intersects each edge of $\mathcal{T}$ at most once.

	We start by showing that the edges of $\mathcal{T}$ in $T$ are Euclidean geodesics. Following the proof of Proposition~\ref{prop:triangulation}, we need to show that, for each point $x$ on $T$, its closest points in $P$ all lie in $T$ and have distance at most $\eps/30$ from $x$. We also need to show that the shortest geodesic joining $x$ to any of these points remains within the cusp. The first of these statements holds by our choice of $P$.
	
	Note that $T$ lies within $M_{(0,\eps]}$. By definition of the Margulis constant, $M_{(0,\eps]}$ consists of a cusp and some regular neighbourhoods of geodesics with length at most $\eps$. The Euclidean metrics on $T$ and the cusp component of $\partial M_{(0,\eps]}$ differ by a Euclidean scale factor of $4/3$, and hence are hyperbolic distance $\ln(4/3) > 0.287$ from each other. On the other hand, the 3-dimensional Margulis constant satisfies $\eps_3 < 0.775$. (See the discussion in \cite[Section 1.1]{FuterPurcellSchleimer}.) Hence, $\eps/30 < \ln(4/3)$. We deduce that for each point $x$ in $T$, any shortest geodesic to a closest point in $P$ must lie in the cusp. This implies that the restriction to $T$ of the Voronoi diagram for $P$ in $M$ is equal to the Voronoi diagram for $P \cap T$ in $T$ with its Euclidean metric. In particular, the edges of $\mathcal{T}$ in $T$ are Euclidean geodesics, as claimed in (\ref{it:short-edges}).

	Let $N_{\mathrm{max}}$ be a maximal cusp neighbourhood around $K$. Then $N_{\mathrm{max}}$ contains $T$. This torus $T$ is a scaled copy of $\partial N_{\mathrm{max}}$. It is scaled so that for each point on $T$, two lifts of this point in $\mathbb{H}^3$ are exactly $3\eps/4$ apart and no two lifts of this point are any closer than this. Say that $d$ is the hyperbolic distance between $T$ and $\partial N_{\mathrm{max}}$. Then the scale factor taking $\partial N_{\mathrm{max}}$ to $T$ is $e^{-d}$. Now the meridian slope on $\partial N_{\mathrm{max}}$ has length at most $6$. Hence, the meridian slope on $T$ has length at most $6 e^{-d}$. So any point on $T$ has two lifts to $\mathbb{H}^3$ that are less than $6e^{-d}$ apart, and therefore, $3\eps/4 \leq 6 e^{-d}$. As in the proof of Proposition~\ref{prop:triangulation}, let $h$ be the length in $\partial N_{\mathrm{max}}$ of a Euclidean geodesic that starts and ends on a geodesic with slope $\lambda - n \mu$ and that is orthogonal to this geodesic. It was shown there that $h \geq 0.55$. Hence, the length of the corresponding geodesic on $T$ is at least $0.55 e^{-d} \geq (0.55/6) (3\eps/4)$. On the other hand, the length of each edge of $\mathcal{T}$ on $T$ is at most $\eps/15$, and $\eps/15 < (0.55/6) (3\eps/4)$. Hence, each such edge can intersect any geodesic with slope $\lambda - n \mu$ at most once. This establishes the claimed properties of $\mathcal{T}$.
		
	Let $T_1, \dots, T_m$ be the components of $\d M_{[3\eps/4, \infty)}$, where $T_i$ encircles a geodesic $\gamma_i \in \mathrm{OddGeo}(\eps/2)$. Let $\tw(\gamma_i) = p \lambda_i + q \mu_i$, where $\lambda_i$ is the canonical longitude on $T_i$ and $\mu_i$ is the meridian, and let $C_i$ be a curve on $T_i$ with this slope.
Then %\marginpar{AJ: note $\lambda_i$ might not be the homological longitude} 
	\[
	\ell(C_i) \le c_5 \Area(T_i) 
	\]
	by Lemma~\ref{Lem:TwistSlopeShort}. Let 
	\[
	C := \bigcup_{i=1}^m C_i. 
	\]
	Realise each $C_i$ as a Euclidean geodesic in $T_i$ missing the vertices of $T_i$, and hence as a normal curve in $T_i$. Since $\ell(C_i) \le c_5 \Area(T_i)$ and by property~\eqref{it:short-edges} of the triangulation $\cT$, the normal representative of $C_i$ intersects each edge of $\mathcal{T}$ at most $c_5' \Area(T_i)$ times for a constant $c_5'$ depending only on $\eps$.
	
	We claim that there is a connected normal curve $C_i'$ in $T_i$ for $i \in \{1, \dots, m\}$ with the following properties:
	\begin{enumerate}
		\item $C_i'$ and $C_i$ are equal in $H_1(T_i; \ZZ_2)$;
		\item $C_i'$ intersects each edge of $\mathcal{T}$ at most once. 
	\end{enumerate}
	This is constructed as follows. For each edge of $\mathcal{T}$ that intersects $C_i$ an odd number of times, replace this intersection by a single point of intersection. These will be the points of intersection between $C_i'$ and the 1-skeleton of $\mathcal{T}$. Since $|C_i \cap \d t|$ is even for each triangle $t$ of $\cT$, we have $|C_i' \cap \d t| \in \{0, 2\}$. If $|C_i' \cap \d t| = 2$, join the two points of $C_i' \cap \d t$ by a normal arc of $C_i'$. The result is a collection of simple closed curves in $T_i$ that are mod 2 homologous to $C_i$. If any of these curves are inessential in $T_i$, remove them. The resulting curves are essential in $T_i$. Since they are non-trivial in mod 2 homology, they consist of an odd number of parallel copies of a curve. If this odd number is greater than one, remove all but one of these curves. The result is $C_i'$, and we write 
	\[
	C' := \bigcup_{i=1}^m C_i'.
	\]
	
	Let $C''$ be the union of $C'$ and a normal curve $C_K$ of slope $(1, -n)$ on $\partial N(K)$, where $n$ is a closest even integer to $\slope(K)$. We claim that $C''$ bounds an unoriented surface in $M_{[3\eps/4, \infty)}$. As $n$ is even, there is a compact surface properly embedded in the exterior of $K$ with boundary slope $(1, -n)$. It intersects each geodesic with length at most $\eps/2$ in a collection of meridians. For a geodesic with odd linking number with $K$, the number of these meridians is odd. For the others, it is even. As $C_i'$ is homologous to the meridian of $T_i$ over $\ZZ_2$, 
	%\marginpar{AJ: $C_i'$ might not be homologous to $\mu_i$ if $\lambda_i$ minus homological longitude is $(2k+1) \mu_i$. ML: $[p \lambda_i + q\mu_i] = [\mu_i]$ in mod 2 homology}
	we may modify the surface so that its boundary is precisely $C''$. This proves the claim.
	
	As $C''$ intersects each edge of $\mathcal{T}$ at most once, we can find a surface $F''$ in $M_{[\eps/2, \infty)}$ that it bounds such that 
	\[
	-\chi(F'') \le c_6 \vol(K) 
	\]
	for some constant $c_6$, just like in the proof of Theorem~\ref{thm:crosscap}. Now $C_i$ and $C_i'$ are equal in $H_1(T_i; \ZZ_2)$. Hence, we may insert a compact connected surface $F_i$ into a regular neighbourhood $N(T_i)$ of $T_i$ with $\d F_i = C_i \cup C_i'$. Since $C_i$ and $C_i'$ intersect each edge of $\mathcal{T}$ at most $c_5' \Area(T_i)$ times, this surface may be chosen so that 
	\[
	-\chi(F_i) \le c_5'' \Area(T_i) 
	\]
	for a constant $c_5''$ depending only on $\eps$. Hence, the surface  
	\[
	F := F'' \cup \bigcup_{i=1}^m F_i \subset M_{[3\eps/4, \infty)}
	\]
	satisfies $\d F = C_K \cup C$, and
	\begin{equation}\label{eqn:chi-F}
	-\chi(F) \le c_6 \vol(K) + \sum_{i=1}^{m} c_5'' \Area(T_i) \le c_7 \vol(K)
	\end{equation}
	for a constant $c_7$ that depends only on $\eps$. Here, the last inequality follows from the observation that $\Area(T_i) \le c_8 \vol(N(T_i))$ for some constant $c_8$, 
	%\marginpar{ML: There were multiple uses of the constant $c$. I have renamed them.}
	where 
	\[
	N(T_i) := \{\, x \in V_i : d(x, T_i) \le r_i/2\,\},
	\]
    and $V_i$ is the solid toral component of $M_{(0,3\eps/4]}$ of tube radius $r_i$ that encloses the geodesic $\gamma_i \in \mathrm{OddGeo}(\eps/2)$.
	
    In each $V_i$, we construct the surface provided by Lemma \ref{Lem:SurfaceSolidTorus} with boundary $C_i = C \cap V_i$. We attach these surfaces to $F$ to form a surface $F_+$. We now specify a basis for $H_1(F_+)$. We start by picking a basis for $H_1(V_1 \cap F_+)$. We arrange that all but one of these basis elements have zero winding number around $V_1$. We then continue to $V_2$, and so on. We then extend this to a basis for $H_1(F_+)$ by adding some oriented curves in $F$. We order this basis as follows into $n+1$ blocks. In the first block, we place all the basis elements of $H_1(V_1 \cap F_+)$ that have zero winding number around $V_1$. 
    %\marginpar{AJ: Is it an issue that only the algebraic winding can be made zero? Also in Lemma~4.7. ML: I have added extra argument in Lemma 4.7.}
    In the second block, we do the same for $V_2$, and so on. In the final block, we place all the remaining basis elements. We saw in the proof of Lemma~\ref{Lem:TwistSlopeShort} that there is a constant $a_0$ depending only on $\eps$ such that $\Area(T_i) \ge a_0$. As $\sum_{i=1}^m \Area(T_i) \le c_8 \vol(K)$, we have 
	\[
	|\mathrm{OddGeo}(\eps/2)| \le c_8 \vol(K) / a_0. 
	\]
	This, together with equation~\eqref{eqn:chi-F}, imply that the number of elements in this final block is bounded above by a linear function of $\vol(K)$.
	
	Let $G$ be the submatrix of the Goeritz form $G_{F_+}$ consisting of the first $n$ blocks. By Lemma~\ref{Lem:AddRowColumn}, $\sigma(G)$ and $\sigma(G_F)$ differ by at most the number of elements in the final block. Note that $G$ is block diagonal. 
	%\marginpar{ML: Note that $G_{F_+}$ is not quite block diagonal.}
	For the block corresponding to $V_i$, the signature differs from $\sigma(G_{F_+ \cap V_i})$ by at most one by Lemma~\ref{Lem:AddRowColumn}. On the other hand,
	\[
	|\sigma(G_{F_+ \cap V_i}) + \kappa(\tw_p(\gamma_i), \tw_q(\gamma_i))| \le 2
	\]
	by Lemma~\ref{Lem:SurfaceSolidTorus}.
	Hence, 
	\[
	\left | \sigma(G_{F_+}) + \sum_{\gamma \in \mathrm{OddGeo}(\eps/2)} \kappa(\tw_p(\gamma_i), \tw_q(\gamma_i)) \right | \le c_9 \vol(K) 
	\]
	for some constant $c_9$.
	By Gordon and Litherland's theorem (Theorem~\ref{thm:Gordon-Litherland}), 
	\[
	\sigma(K) = \sigma(G_{F_+}) + e(F_+)/2 = \sigma(G_{F_+}) + n / 2. 
	\]
	The result follows as $n$ is a closest even integer to $\slope(K)$. 
\end{proof}

\section{Experimental data and some conjectures about random knots}
\label{sec:conjectures}

We set out to find links between hyperbolic and 4-dimensional knot invariants. Initial scatter plots compared some 4-dimensional invariants (the signature and Heegaard Floer invariants $\tau$, $\nu$, and $\eps$),
the crossing number, and several hyperbolic invariants (volume, meridional and longitudinal translations, and the Chern--Simons invariant). As $\sigma$ is strongly correlated to $\tau$, $\nu$, and $\eps$, we decided to only focus on $\sigma$, which is more classical and easier to compute.

The strongest and most surprising correlation was between the signature and the real part of the meridional translation; see Figure~\ref{fig:butterfly}. There were some more predictable relationships among the hyperbolic invariants.

\begin{figure}
	\includegraphics[scale = 0.5]{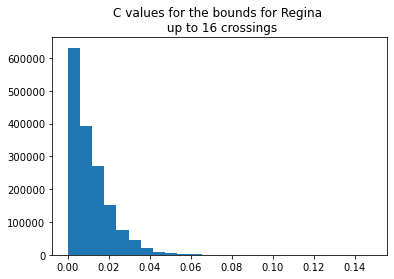}
	\caption{The distribution of 
		\[
		c_1(K) := |2 \sigma(K) - \slope(K)| \inj(K)^3 / \vol(K) 
		\]
		for knots up to 16 crossings in the Regina census.}
	\label{fig:c-values}
\end{figure}

\begin{figure}
	\includegraphics[width = 0.49 \textwidth]{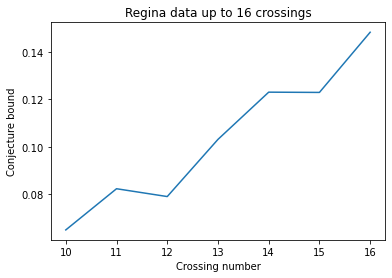}
	\includegraphics[width = 0.49 \textwidth]{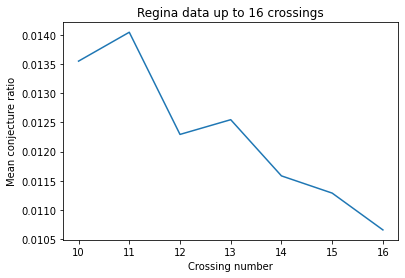}
	\caption{The maximum (left) and the mean (right) of $c_1(K)$
		as functions of the crossing number
		for knots up to 16 crossings in the Regina census.}
	\label{fig:c-by-crossing}
\end{figure}

Figure~\ref{fig:c-values} shows the distribution of 
\[
c_1(K) := |2 \sigma(K) - \slope(K)| \inj(K)^3 / \vol(K), 
\]
which indicates that the constant $c_1$ appearing in Theorem~\ref{thm:main} is typically quite small. The largest value of this quantity we managed to obtain is less than $0.234$, and we conjecture it is always at most 0.3. The left of Figure~\ref{fig:c-by-crossing} shows the maximum  and the right the mean of $c_1(K)$ by crossing number for the Regina census of knots of at most 16 crossings. 
See Figure~\ref{fig:injectivity-vol} for a scatter plot of injectivity radius versus volume for random hyperbolic knots of 10-80 crossings. This suggests that the injectivity radius is typically not too small as the volume increases.

\begin{figure}
	\includegraphics[scale = 0.3]{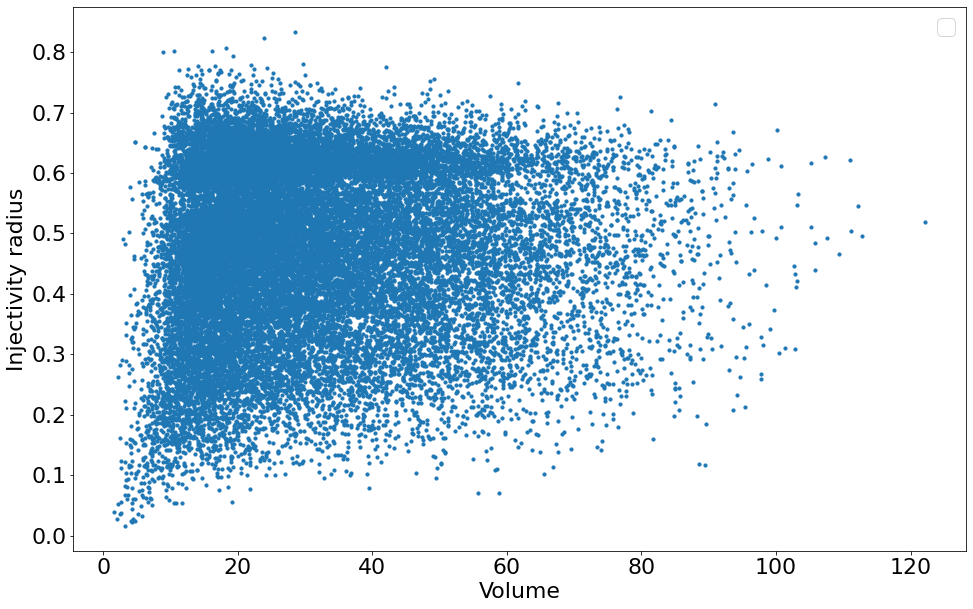}
	\caption{A scatter plot of injectivity radius versus volume for random knots of 10-80 crossings.}
	\label{fig:injectivity-vol}
\end{figure}

We will say that a property $P$ holds asymptotically almost surely, or a.a.s., in short, if the probability that $P$ holds for knots of $n$ crossings tends to 1 as $n \to \infty$.

It is known that there is a constant $A$ such that $\vol(K) \le A\, c(K)$, where $c(K)$ is the crossing number of $K$. From scatter plots, one might conjecture that there is a constant $a$ such that $a \, c(K) \le \vol(K)$ a.a.s. Such an inequality cannot hold for all hyperbolic knots $K$. For example, consider twist knots. More generally, the highly twisted knots considered in Section \ref{sec:highly-twisted} have bounded volume but unbounded crossing number.
%\marginpar{ML: Dropped constant $b$ because it's unnecessary. Reworded}

We now consider the behaviour of the signature $\sigma(K)$ for random knots $K$. 
By Theorem~\ref{thm:Gordon-Litherland}, $\sigma(K)$ can be computed from the black surface of a checkerboard colouring of a diagram of $K$. Hence, it is the signature of a $c(K) \times c(K)$ matrix. If the signs of the eigenvalues of this matrix were independently distributed, then the expected value of $|\sigma(K)|$ would be $C \sqrt{c(K)}$ for some constant $C$. From computational evidence, it appears the constant is about $2$. 
 Based on this heuristic, we introduce the following definition:

\begin{definition}
The \emph{normalised signature} of a hyperbolic knot $K$ is
\[
\sh(K) := \frac{\sigma(K)}{\sqrt{\vol(K)}}.
\]
\end{definition}

We use the volume instead of the crossing number as it is easier to compute using SnapPy and is more regular.

Based on Figure \ref{fig:butterfly}, we initially conjectured that for any hyperbolic knot $K$ in $S^3$ with $|\widehat{\sigma}(K)| > 1$, the signature $\sigma(K)$ and $\mathrm{Re}(\mu(K))$ have the same sign. However, this turns out not to be true.

\begin{corollary}\label{cor:counterexample2}
	There exists a hyperbolic knot $K$ with $|\widehat{\sigma}(K)| > 1$, but with $\sigma(K)$ and $\mathrm{Re}(\mu(K))$ having opposite signs.
\end{corollary}

\begin{proof}
	We start with a hyperbolic link $K \cup C_1 \cup C_2$ in $S^3$ where $C_1$ and $C_2$ bound disjoint embedded discs, and where $\ell_1 = \lk(K, C_1) = 2$ and $\ell_2 = \lk(K,C_2) = 3$. We then build the highly twisted knots $K(q_1, q_2)$ as in Theorem~\ref{Thm:SignatureSlopeTwisting}. Set $q_1 = 17q$ and $q_2 = -8q$, where $q$ is a large positive integer. Then 
	\[
	\begin{split}
	\slope(K(q_1, q_2)) &\sim -4 \cdot 17 q + 9 \cdot 8 q = 4q \text{, whereas} \\ 	
	\sigma(K(q_1, q_2)) &\sim -2 \cdot 17q + 4 \cdot 8q = -2q. 
	\end{split}
	\]
	Hence, for $q$ sufficiently large, $\sigma(K(q_1, q_2))$ and $\slope(K(q_1,q_2))$ have opposite signs, and hence $\sigma(K(q_1,q_2))$ and $\mathrm{Re}(\mu(K(q_1, q_2)))$ also have opposite signs by Lemma~\ref{lem:slope-formula}. Note that $\widehat{\sigma}(K(q_1, q_2)) > 1$ if $q$ is sufficiently large, because $|\sigma(K(q_1, q_2))|$ tends to infinity whereas $\vol(K(q_1, q_2))$ is bounded.
\end{proof}

However, we do conjecture the following:

\begin{conjecture}
	\label{Conj:Signs}
	If $K$ is a hyperbolic knot in $S^3$ with $|\widehat{\sigma}(K)| > 1$, then $\sigma(K)$ and $\mathrm{Re}(\mu(K))$ have the same sign asymptotically almost surely.
\end{conjecture}

%By Corollary~\ref{cor:counterexample2}, this does knot hold for all knots.
%We note that the sign of $\slope(K)$ equals the sign of $-\mathrm{Re}(\mu)$. So this conjecture can be reinterpreted as saying that $\sigma(K)$ and $\slope(K)$ have the same signs provided that $|\widehat{\sigma}(K)| > 1$.

We also state the following conjecture, which proposes a more precise relationship between slope and signature.

\begin{conjecture}\label{conj:asymptotic}
	There are constants $b$ and $c$ such that, for any hyperbolic knot $K$ in $S^3$, we have
	\begin{equation} \label{eqn:conj}
		|2\sigma(K) - \slope(K)| \le b\sqrt{\vol(K)} + c
	\end{equation}
	asymptotically almost surely.
\end{conjecture}

By Corollary~\ref{cor:counterexample1}, this does not hold for all knots either. In fact, there are families of hyperbolic knots for which $|2\sigma(K) - \slope(K)|$ is not bounded by a linear function of the volume.

\begin{comment}
This does not hold for all knots: Consider closures $K_n$ of 4-strand braids of the form $(\sigma_1 \sigma_2^{-1} \sigma_3)^n$ for $n$ odd. These are alternating, and have signature $\sigma(K_n) = 1 - n$. Hence, $\lim_{n \to \infty} \sigma(K_n) / c(K_n) = -1/3$. On the other hand, 
\[
\lim_{n \to \infty} \frac{\slope(K_n)} {c(K_n)}  \approx 0.7715321.
\]
As $\vol(K_n)  \le Ac(K_n)$, we have $\lim_{n \to \infty} \sqrt{\vol(K_n)} / c(K_n) = 0$. So this family of knots violates equation~\eqref{eqn:conj}.

More generally, if $K_n$ is the closure of a braid $w^n$ for some braid word $w$, then $\lim_{n \to \infty} \slope(K_n) / n = c_1$ for some constant $c_1$. When $K_n$ is alternating (this is the case when the exponent of $\sigma_i$ is positive for $i$ odd and is negative for $i$ even), $c(K_n)$ is the length of $w^n$, except for some trivial counterexamples.  On the other hand, using the formula for the signature of alternating knots, $\lim_{n \to \infty} \sigma(K_n) / n  = c_2$ for another constant $c_2$. So 
\[
\lim_{n \to \infty} |2\sigma(K_n) / c(K_n) - \slope(K_n)/c(K_n)| \neq 0
\]
whenever $c_1 \neq 2c_2$, which is the case in the above example, where $w = \sigma_1 \sigma_2^{-1} \sigma_3$.
\end{comment}

The proof of Theorem~\ref{thm:main} provides some heuristic for Conjecture~\ref{conj:asymptotic}. Indeed, if we assume that the signs of the eigenvalues of the Goeritz matrix $G_F$ are independent, 
%\marginpar{ML: The eigenvalues are probably not normally distributed but their signs are just $-1$ or $1$. Changed $\sqrt{|G_F|}$to $\sqrt{c(K)}$ as $|G_F|$ often means determinant}
then the signature on average is of order $\sqrt{c(K)}$. This justifies the factor $\sqrt{\vol(K)}$ in the upper bound.

\begin{figure}
	\includegraphics[scale = 0.5]{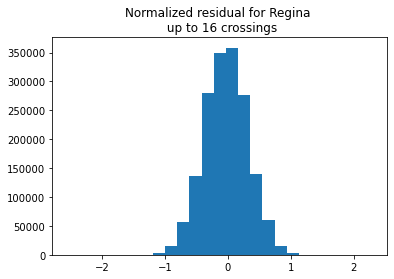}
	\caption{The distribution of 
		\[
		(2 \sigma(K) - \slope(K))  / \sqrt{\vol(K)} 
		\]
		for knots up to 16 crossings in the Regina census.}
	\label{fig:c-distribution}
\end{figure}

If $b < 2$ (and the data supports this; see Figure~\ref{fig:c-distribution}), then Conjecture~\ref{conj:asymptotic} implies Conjecture~\ref{Conj:Signs} for knots $K$ with sufficiently large volume a.a.s. This is because equation~\eqref{eqn:conj} is equivalent to the inequality
\[
\left| 2 \widehat{\sigma}(K) - \left(\slope(K)/\sqrt{\vol(K)} \right) \right | \leq b + c / \sqrt{\vol(K)}.
\]
If $b < 2$, then $b + c / \sqrt{\vol(K)} < 2$ for all knots with sufficiently large volume. So, if $|\widehat{\sigma}(K)| > 1$, then $\widehat{\sigma}(K) $ and $\slope(K)$ have the same sign.

\begin{comment}
The choice of coefficient $2$ in front of $\sigma(K)$ is justified by the experimental data and also by the knots in Theorem~\ref{Thm:SignatureSlopeTwisting}.

Finally, we remark that the conjectures also seem to hold for links, and the error is linear in the
number of link components. The correlation is better with the Gordon--Litherland correction term for
the signature.
\end{comment}

\bibliography{signature-cusp-geometry-biblio}
\bibliographystyle{plain}

\end{document}